\documentclass[10pt]{amsart}

\usepackage{soul}

\usepackage[latin2]{inputenc}
\usepackage{amsmath}
\usepackage{graphicx}
\usepackage{amssymb}
\usepackage{esint}
\usepackage{color}
\usepackage{amsthm}
\usepackage{epsfig}
\usepackage{enumitem}
\usepackage{mathtools}
\usepackage{esvect}
\usepackage{mathrsfs}

\usepackage{longtable}

\usepackage[dvipsnames]{xcolor}
\usepackage{tikz}
\usepackage{xxcolor}

\usepackage{tikz}
\usetikzlibrary{calc,intersections,through,backgrounds,patterns,arrows.meta, math,through}
\usepackage{tkz-euclide}
\usepackage[labelformat = brace, position=top]{subcaption}
\usepackage[english]{babel}
\newtheorem{theorem}{Theorem}

\newtheorem{lemma}[theorem]{Lemma}

\newtheorem{remark}[theorem]{Remark}
\newtheorem*{theorem*}{Theorem}

\theoremstyle{definition}

\newcommand{\Rd}[1][d]{{\mathbb{R}^{#1}}}
\allowdisplaybreaks[4]

\newcommand{\dH}[1][d-1]{\,\mathrm{d}\mathcal{H}^{#1}}
\newcommand{\dx}{\,\mathrm{d}x}

\newcommand{\dS}{\,\mathrm{d}S}
\newcommand{\ds}{\,\mathrm{d}s}

\newcommand{\dvectheta}{\,\mathrm{d}\vec{\theta}}

\newcommand{\dvectildetheta}{\,\mathrm{d}\vec{\theta'}}
\newcommand{\n}{\mathrm{n}}
\newcommand{\ta}{\mathrm{t}}
\newcommand{\eps}{\varepsilon}
\newcommand{\supp}{\mathrm{supp}\,}
\newcommand{\dist}{\mathrm{dist}}

\usepackage[linktocpage=true,colorlinks=true]{hyperref}

\begin{document}

\title[A quantitative isoperimetric inequality via calibrations]
{The quantitative isoperimetric inequality: A calibration argument}

\author{Sebastian Hensel}
\address{Universit{\"a}t Leipzig, Fakult\"at f\"ur Mathematik und Informatik, Neues Augusteum, Augustusplatz 10, 04109 Leipzig, Germany}
\email{sebastian.hensel@uni-leipzig.de}

\author{Tim Laux}
\address{Universit\"at Heidelberg, Institut f\"ur Mathematik \& Interdisziplin\"ares Zentrum f\"ur Wissenschaftliches Rechnen, Mathematikon, Im Neuenheimer Feld 205, 69120 Heidelberg, Germany}
\email{tim.laux@math.uni-heidelberg.de}

\thanks{This project was initiated during a stay of both authors at the Erwin Schr\"odinger International Institute for Mathematics and Physics (ESI)
for the workshop ``Geometric Variational Problems'' in the framework of the thematic program ``Free Boundary Problems''. 
We warmly thank ESI for the support and hospitality. The authors are also grateful to Marco Pozzetta for insightful discussions
on the subjects of the paper and, in particular, for suggesting the simple contradiction argument for a proof of Theorem~\ref{theo:globalRegime}
based on Theorem~\ref{theo:pertRegime}.
We also thank Guido De~Philippis for pointing out some relevant references.}%

\begin{abstract}
	We give a short proof of the quantitative isoperimetric inequality.
	Our argument is based on a notion of quantitative calibrations which induce a natural distance controlling both the Fraenkel asymmetry and the tilt excess. 
	The proof of our key result which can be viewed as a nonlinear, geometric version of Fuglede's result in $BV$ is direct and self-contained. 
	In particular, we do not make any use of regularity theory for almost minimizers.
\end{abstract}

\keywords{%
}%

\maketitle
\tableofcontents

\section{Introduction}

\subsection{Context and motivation}

The isoperimetric inequality is the answer to one of the oldest problems in mathematics stating that, in Euclidean space, among all measurable sets of fixed volume, the ball minimizes the perimeter, or equivalently\footnote{Here, by scaling invariance, we restrict ourselves to the case of the unit ball $B_1$.}
\begin{equation}\label{eq:Isoperimetric Inequality}
	0\leq P(F)-P( B_1) \quad \text{for all } F\subset \Rd \text{ with } |F|=|B_1|.
\end{equation}
While the first arguments go back to Ancient Greece, it was not until the work of De~Giorgi~\cite{DeGiorgi_isoperimetric} in the 1950's that the problem was completely solved and a full proof among arbitrary measurable sets was found.
Moreover, De~Giorgi's  symmetrization argument also shows that the only sets that satisfy equality in~\eqref{eq:Isoperimetric Inequality} are balls (up to sets of measure zero).

This suggests the natural question of \emph{stability}, namely: 
If the isoperimetric deficit $P(F)-P(B_1)$ is small, is $F$ close to a ball in a suitable topology?
The answer can be elegantly formulated in terms of a quantitative version of the isoperimetric inequality~\eqref{eq:Isoperimetric Inequality} of the form
\begin{equation}\label{eq:Quant Isoperimetric Inequality}
	  \frac1C \inf_{x_0\in\Rd} \rho^2(F{-}x_0,B_1)  \leq P(F)-P( B_1)
	\quad \text{for all } F\subset \Rd \text{ with } |F|=|B_1|,
\end{equation}
where $\rho$ is a suitable ``distance'' between sets and the constant $C=C(d)$ is independent of $F$.
The first complete quantitative isoperimetric inequality was proven by Fusco, Maggi, and Pratelli~\cite{Fusco2008} in 2008, and since then a few alternative proofs have been found.
Figalli, Maggi, and Pratelli~\cite{FigalliMaggiPratelli2010} used optimal transport instead of
symmetrization, quantifying an argument laid out by Gromov~\cite[Appendix]{Milman86_Gromov}.
Cicalese and Leonardi~\cite{Cicalese2012} found the optimal constant based on a selection principle and regularity theory for almost-minimizers of the perimeter.
All of these proofs use the Fraenkel asymmetry on the left-hand side of~\eqref{eq:Quant Isoperimetric Inequality}, i.e., the $L^1$-distance $ \rho(F,B_1) = |F\Delta B_1|$.
Fusco and Julin obtained a lower bound in terms of the tilt-excess $\rho^2(F,B_1) = \int_{\partial^* F} |\n_{\partial^\ast F} - \n_{\partial B_1}|^2 \dH$ in~\cite{FuscoJulin14}.

\subsection{Our contribution}

In the present work, we propose an alternative, direct proof of the quantitative isoperimetric inequality based on the notion of \emph{calibrations}.
Our ``quantitative calibration'' is a vector field $\xi$ that is designed to be an almost calibration for the ball $B_1$ and, at the same time, to penalize deviations of any competitor $F$ from $B_1$ in a suitable way.
The construction of $\xi$ is straight-forward in our case of the ball, see Figure~\ref{fig:calibration}, and we expect this to be possible also in more complex problems, for example isoperimetric clusters, by combining our ideas here with those from~\cite{fischer2023}.
This quantitative calibration allows us to quantify the error in the calibration argument for minimality precisely and in the natural topology.
The previous results on the quantitative isoperimetric inequality with the Fraenkel asymmetry~\cite{Fusco2008,FigalliMaggiPratelli2010} or the tilt-excess~\cite{FuscoJulin14} are then direct corollaries.

More precisely, our new notion of quantitative calibration $\xi$ naturally induces the following excess-type functional
\[
	\rho^2(F,B_1)=E_{rel}[F|\xi] = \int_{\partial^\ast F} (1-\xi\cdot \n_{\partial^\ast F})\dH.
\]
The main novelty of the present work is Theorem~\ref{theo:pertRegime} below, which, for any $\gamma>0$, provides the sharp lower bound
\[
	(1-\gamma)\frac12 \rho^2(F,B_1) \leq P(F)- P(B_1)
\]
in the perturbative regime $\rho(F,B_1)\ll_\gamma 1$, and with $F$ having the same barycenter as $B_1$. 
This can be viewed as a nonlinear, geometric Fuglede-type result in $BV$. 
Note that, unlike Fuglede~\cite{Fuglede86}, we do not require additional assumptions like graph-structure or smoothness on $F$, and we obtain the optimal constant $\big(\tfrac12\big)^-$.

In the proof, starting from an arbitrary set of finite perimeter $F$ with $\rho^2(F,B_1) \ll1$, by a slicing argument, we first replace it by a $BV$ graph, then by a smooth graph, and then conclude by a version of Fuglede's classical result, see Figure~\ref{fig:graph_approximation}. 
Here, the key is to estimate all errors in the approximations and the spectral argument by our functional $\rho^2(F,B_1)=E_{rel}[F|\xi]$.
To obtain our global result, Theorem~\ref{theo:globalRegime}, we only need to combine this local result with a standard compactness argument in the case $E_{rel}[F|\xi] \gtrsim1$.

The idea of using calibrations to prove the isoperimetric inequality~\eqref{eq:Isoperimetric Inequality} was already suggested by H\'elein~\cite{MR1306553}. 
However, calibration arguments have only been used for related, simpler problems in which the isoperimetric profile is convex, namely in the stability of the Simon's cone~\cite{DePhilippisMaggi} or in spectral shape optimization~\cite{BrascoDePhilippisRuffini}.
The beautiful theory of calibrated geometry has a long history, dating back to Wirtinger~\cite{Wirtinger36} and de~Rham~\cite{deRham57}.
Crucial to the present work are recent techniques from~\cite{FischerHensel2020, FHLS20, fischer2023}.
One of the main advantages of such a calibration argument is that it is direct and does not rely on heavy regularity theory.
In fact, the most technical aspect of our proof is a slicing argument.
We are optimistic that variants of this calibration argument can be applied to other geometric settings and in particular to isoperimetric clusters.

Most recently, in an independent work~\cite{figalli2026sharpstabilityalexandrovstheorem}, Figalli and Zhang prove a similar $BV$ Fuglede result in the graph setting with different techniques.

\subsection{Notation} 
For two sets $A,B$ we denote their symmetric difference by $A\Delta B$.
The open ball of radius $R>0$ and centered at the origin is given by $B_R$, with $\partial B_R$ denoting its boundary.
For a measurable set~$A \subset \Rd$, we write $\chi_{A}$ for its associated indicator function.
If $F \subset \Rd$ is a set of finite perimeter, $\partial^*F$ refers to its reduced boundary, in particular $P(F)=\mathcal{H}^{d-1}(\partial^\ast F)$. 
We denote its measure theoretic (interior) unit normal by $\n_{\partial^\ast F} \colon \partial^\ast F \to \mathbb{S}^{d-1}$.
The $d$-dimensional Lebesgue measure and the $(d{-}1)$-dimensional Hausdorff measure are
denoted by $\mathcal{L}^d$ and $\mathcal{H}^{d-1}$. We also make use of standard
notation for Lebesgue and Sobolev spaces: $L^p$ and $W^{k,p}$.

\section{Main results}
We define our notion of a ``quantitative calibration of the unit ball~$B_1$''
and state our version of the quantitative isoperimetric inequality for the associated excess.

\begin{figure}
	\begin{subfigure}{0.3\textwidth}
		\centering
		\begin{tikzpicture}[scale=1.4]
			
			\draw[red,thick]  (0,0)  circle ({1});

			
			\draw[blue, thick, domain=0:2*pi,smooth,samples=200] plot ({deg( \x) }:{(2-0.6*(cos(deg(3*\x)))^2)/1.73});
		\end{tikzpicture}
	\end{subfigure}
	\begin{subfigure}{0.3\textwidth}
		\centering
		\begin{tikzpicture}[scale=1.4]
		
			\draw[red,thick]  (0,0)  circle ({1});

			\def \xmax{20}
			\foreach \x in {1, ..., \xmax}
				\foreach \r in {  1, 1.3}
				{
				\draw[->] ({deg( 2*pi*\x/(\xmax))}: {\r}) -- + ({deg( 2*pi*\x/(\xmax))}:{-(1-(\r-1)^2)*0.25});
				}
			\def \xmax{20}
			\foreach \x in {1, ..., \xmax}
				\foreach \r in { 0.7}
				{
				\draw[->] ({deg( 2*pi*\x/(\xmax))}: {\r}) -- + ({deg( 2*pi*\x/(\xmax))}:{-(1-(\r-1)^2)*0.25});
				}
			\def \xmax{10}
			\foreach \x in {1, ..., \xmax}
				\foreach \r in { 0.4}
				{
				\draw[->] ({deg( 2*pi*\x/(\xmax))}: {\r}) -- + ({deg( 2*pi*\x/(\xmax))}:{-(1-(\r-1)^2)*0.25});
				}
			\def \xmax{5}
			\foreach \x in {1, ..., \xmax}
				\foreach \r in { 0.2}
				{
				\draw[->] ({deg( 2*pi*\x/(\xmax))}: {\r}) -- + ({deg( 2*pi*\x/(\xmax))}:{-(1-(\r-1)^2)*0.25});
				}
		\end{tikzpicture}
	\end{subfigure}
	\begin{subfigure}{0.3\textwidth}
		\centering
		\begin{tikzpicture}[scale=1.4]
			
			\draw[red,thick]  (0,0)  circle ({1});

			\def \xmax{20}
			
			\draw[blue, thick, domain=0:2*pi,smooth,samples=200] plot ({deg( \x) }:{(2-0.6*(cos(deg(3*\x)))^2)/1.73});
			\foreach \x in {1, ..., \xmax}
			{
				\draw[->]({deg( 2*pi*\x/(\xmax))}: {(2-0.6*(cos(deg(3*2*pi*\x/(\xmax))))^2)/1.73}) -- + ({deg( 2*pi*\x/(\xmax))}:{-(1-(((2-0.6*(cos(deg(3*2*pi*\x/(\xmax))))^2)/1.73)^(0.5)-1)^2)*0.25});
			}
		\end{tikzpicture}
	\end{subfigure}
	\caption{A set of finite perimeter $F$ and a ball $B$ of the same volume (left). 
	To measure the closeness of the two, we define the quantitative calibration $\xi$ (center); then we integrate $\xi$ (against the measure theoretic normal $\n_{\partial^\ast F}$) along the reduced boundary $\partial^\ast F$ (right).}
	\label{fig:calibration}
\end{figure}
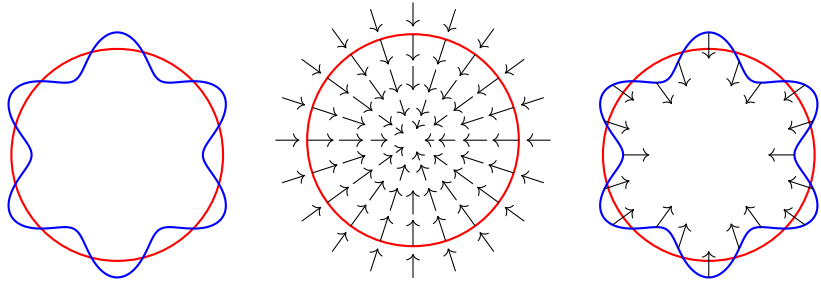

\begin{theorem}
\label{theo:globalRegime}
Let $d \geq 2$. We define a vector field
$\xi \in W^{1,\infty}(\Rd;\Rd)$ by 
\begin{align}
\label{eq:quantCalibration2}
\xi(x) :=  \max\big\{1 - \big(1 {-} |x|\big)^2, 0\big\} \Big(-\frac{x}{|x|}\Big)
\end{align}
and call it a \emph{quantitative calibration of~$B_1$}.
For a set of finite perimeter~$F \subset \Rd$, we define the associated
relative energy $E_{rel}[F|\xi]$ by 
\begin{align}
\label{def:relativeEnergy}
E_{rel}[F|\xi] := \int_{\partial^*F} 1 - \n_{\partial^*F} \cdot \xi \dH.
\end{align}

Then, there exists a constant $C=C(d) \in (1,\infty)$ 
such that for all sets of finite 
perimeter~$F \subset \Rd$ satisfying
\begin{align}
\label{eq:volumeConstraint1}
\mathcal{L}^d(F) &= \mathcal{L}^d(B_1) =: \omega_d
\end{align}
it holds
\begin{align}
\frac{1}{C} \inf_{x_0 \in \Rd} E_{rel}[F{-}x_0|\xi] \leq \mathcal{H}^{d-1}(\partial^* F) - \mathcal{H}^{d-1}(\partial B_1).
\end{align}
\end{theorem}

The main contribution of the present paper consists of the following ``Fuglede-type''
result. The novelty lies in the formulation of the allowed perturbation: whereas
Fuglede's classical result~\cite{Fuglede86} assumes $C^1$-closeness of the 
competitor interface to the unit sphere, we only require smallness of excess~\eqref{def:relativeEnergy} 
with respect to the quantitative calibration~$\xi$ of~$B_1$, see~\eqref{eq:quantCalibration2}. We emphasize 
that in our version the same measure of distance is used to quantify both the 
perturbative nature of the competitor and the resulting stability in the isoperimetric problem.
This is in strong contrast to Fuglede's work transferring $C^1$-smallness of the
height function parametrizing the interface of the perturbative competitor to
a lower bound of the isoperimetric deficit in terms of the $W^{1,2}$-norm of the height function.
Apart from being interesting in its own right, this allows us to completely bypass
the usage of any type of regularity theory to upgrade the perturbative statement of Theorem~\ref{theo:pertRegime}
into the quantitative isoperimetric inequality of Theorem~\ref{theo:globalRegime}:
indeed, excess~\eqref{def:relativeEnergy} is stable under strict convergence
of sets of finite perimeter as the correction to the perimeter of the competitor
is a bulk term.

\begin{theorem}
\label{theo:pertRegime}
Let $d \geq 2$ and $R \in [2,\infty)$. 
For all $\gamma \in (0,1)$ one may choose $\varepsilon = \varepsilon(d,R) \ll \gamma$ 
such that for all sets of finite 
perimeter~$F \subset \Rd$ satisfying
\begin{align}
\label{eq:volumeConstraint2}
\mathcal{L}^d(F) &= \mathcal{L}^d(B_1) = \omega_d,
\\
\label{eq:barycenterConstraint2}
\int_{F} x \,dx &= 0,
\\
\label{eq:diameterConstraint}
F &\subset B_R,
\end{align}
it holds:

If the relative energy $E_{rel}[F|\xi]$ defined by~\eqref{def:relativeEnergy} is small in the sense that
\begin{align}
\label{eq:smallRelEnergy}
E_{rel}[F|\xi] \leq \eps,
\end{align}
then
\begin{align}
\label{eq:quantIsoperInequ2}
(1{-}\gamma)\frac{1}{2} E_{rel}[F|\xi]
&\leq \mathcal{H}^{d-1}(\partial^* F) - \mathcal{H}^{d-1}(\partial B_1).
\end{align}
The constant~$(\frac{1}{2})^{-}$ is sharp with respect to the given $\xi$.
\end{theorem}

The coercivity of the excess~\eqref{def:relativeEnergy} entails
that the quantitative isoperimetric inequalities obtained
by Fusco, Maggi, and Pratelli~\cite{Fusco2008} in terms of Fraenkel asymmetry
or Fusco and Julin~\cite{FuscoJulin14} in terms of ``classical excess'' 
are simple corollaries of Theorem~\ref{theo:globalRegime}.

\begin{lemma}
\label{lem:coercivity}
Let $d \geq 2$. There exists a constant $C=C(d) \in (1,\infty)$ such that
for all sets of finite perimeter~$F \subset \Rd$ satisfying~\eqref{eq:volumeConstraint1} it holds
\begin{align}
\label{eq:controlFraenkel}
\big|\mathcal{L}^d(F \Delta B_1)\big|^2 \leq C E_{rel}[F|\xi]
\end{align}
and
\begin{align}
\label{eq:tiltExcessControl}
\int_{\partial^*F} \frac{1}{2}\Big|\n_{\partial^*F}(x)
- \Big(-\frac{x}{|x|}\Big)\Big|^2 \dH(x)
\leq C E_{rel}[F|\xi],
\end{align}
where the relative energy $E_{rel}[F|\xi]$ is defined by~\eqref{def:relativeEnergy}.
\end{lemma}

We finally state the entry point of our analysis consisting of an identity (!)
relating the difference in perimeters to the excess and a bulk term. We
call this identity the relative energy equality and it is the very reason
why we interpret our approach as a calibration-type argument. The main differences
to the usual calibration computation are two-fold and interrelated: i) we do
not estimate from below when passing from perimeter of the competitor to
its approximation in terms of flux, and ii) we do not require the resulting
bulk term to vanish. Rather to the contrary, our strategy relies on i)
making the error between perimeter of the competitor and its flux approximation
as coercive as possible so that ii) the appearing bulk term (actually having the
dangerous sign for perturbative graph competitors) can be absorbed into a fraction
of the excess.      

\begin{lemma}
\label{lem:relativeEnergyEquality}
Let $d \geq 2$, and for given $a \in\Rd[]$ and~$\vec{b} \in \Rd$,
let $L_{a,\vec{b}}\colon\Rd\to\Rd[]$ be
the affine map $L_{a,\vec{b}}(x) := a + \vec{b}\cdot x$. 
Let $F \subset \Rd$ be a set of finite perimeter satisfying
the volume constraint~\eqref{eq:volumeConstraint2}
and the barycenter constraint~\eqref{eq:barycenterConstraint2}.
Then
\begin{equation}
\label{eq:relativeEnergyEquality}
\begin{aligned}
\mathcal{H}^{d-1}(\partial^* F) - \mathcal{H}^{d-1}(\partial B_1) 
&= E_{rel}[F|\xi] 
- \int_{\Rd} (\chi_F - \chi_{B_1}) \big(\nabla\cdot \xi + (d {-} 1)\big) \dx
\\&~~~
-\int_{\Rd} (\chi_F - \chi_{B_1}) L_{a,\vec{b}} \dx,
\end{aligned}
\end{equation}
where the relative energy $E_{rel}[F|\xi]$ is defined by~\eqref{def:relativeEnergy}.
\end{lemma}

\noindent
\textbf{Structure of the paper.}
In Section~\ref{subsec:relEnergyEquality}, we present the short proof of the
relative energy equality~\eqref{eq:relativeEnergyEquality}. In Section~\ref{subsec:classicalFuglede},
we use it to give a proof of Fuglede's classical result to highlight the spirit of our methodology. 
Sections~\ref{subsubsec:coercivity}--\ref{subsec:proofTheoremPertRegime} contain the proof of
our main contribution, Theorem~\ref{theo:pertRegime}. The short proofs of Theorem~\ref{theo:globalRegime}
and Lemma~\ref{lem:coercivity} are presented in Sections~\ref{subsec:proofMainTheorem} 
and~\ref{subsec:proofCoercivity}, respectively.

\section{Another look at Fuglede's result}
\label{subsec:relEnergyEquality}

\subsection{The relative energy equality: Proof of Lemma~\ref{lem:relativeEnergyEquality}} 
We compute
\begin{equation}
\begin{aligned}
\mathcal{H}^{d-1}(\partial^* F) &=
\int_{\partial^*F} 1 - \n_{\partial^*F}\cdot\xi \dH  + \int_{\partial^*F} \n_{\partial^*F}\cdot\xi \dH
\\&
= E_{rel}[F|\xi] - \int_{\Rd} \chi_{F} \nabla\cdot\xi \dx
\\&
= E_{rel}[F|\xi] - \int_{\Rd} (\chi_{F} {-} \chi_{B_1}) \nabla\cdot\xi \dx
+ \int_{\partial B_1} \n_{\partial B_1} \cdot \xi \dH
\\&
= E_{rel}[F|\xi] - \int_{\Rd} (\chi_{F} {-} \chi_{B_1}) \nabla\cdot\xi \dx
+ \mathcal{H}^{d-1}(\partial B_1).
\end{aligned}
\end{equation}
This, however, is equivalent to~\eqref{eq:relativeEnergyEquality} as a 
consequence of the volume constraint~\eqref{eq:volumeConstraint2} and the centering of the 
barycenter~\eqref{eq:barycenterConstraint2}. \qed

\subsection{A proof of Fuglede's result by the relative energy equality}
\label{subsec:classicalFuglede}
In order to introduce the reader to the relative energy~\eqref{def:relativeEnergy},
we revisit Fuglede's classical result based on the relative energy equality~\eqref{eq:relativeEnergyEquality}.
To this end, we consider a set of finite perimeter~$F \subset \Rd$ 
satisfying~\eqref{eq:volumeConstraint2}--\eqref{eq:barycenterConstraint2} such that
the reduced boundary of~$F$ is given by a $W^{1,\infty}(\partial B_1)$ height function
$h \colon \partial B_1 \to \Rd[]$:
\begin{align}
\label{eq:FugledeAux1b}
\partial^*F = \big\{x + h(x)\n_{\partial B_1}(x) \colon x \in \partial B_1 \big\},
\end{align}
where $\n_{\partial B_1}(x) = -\frac{x}{|x|}$. We claim that for all $\gamma \in (0,1)$
there exists $\varepsilon \ll_\gamma 1$ such that $\|h\|_{W^{1,\infty}(\partial B_1)} \leq \eps$ implies
\begin{equation}
\begin{aligned}
\label{eq:FugledeAux2b}
\mathcal{H}^{d-1}(\partial^* F) - \mathcal{H}^{d-1}(\partial B_1)
&\geq (1{-}\gamma)\frac{1}{2} E_{rel}[F|\xi]
\\&
\geq (1{-}\gamma)^2 \frac{1}{2} \bigg(\int_{\partial B_1} h^2 + \frac{1}{2}|\nabla_{\partial B_1} h|^2 \dH\bigg).
\end{aligned}
\end{equation}

We start by expanding the relative energy. With the change of
variables $\Psi^h\colon \partial B_1 \to \partial^*F$,
$x \mapsto x - h(x) \n_{\partial B_1}(x)$, we obtain from the 
area formula
\begin{equation}
\begin{aligned}
\label{eq:FugledeAux3b}
E_{rel}[F|\xi]
&= \int_{\partial B_1} \frac{(1-h)^{d-1}}{\big|(\n_{\partial^*F}\circ\Psi^h)\cdot\n_{\partial B_1}\big|}
\big(1 - (\n_{\partial^*F}\cdot\xi)\circ\Psi^h \big) \dH
\\&
= \int_{\partial B_1} (1-h)^{d-1} \sqrt{1 + \Big(\frac{|\nabla_{\partial B_1} h|}{1-h}\Big)^2}
\big(1 - (\n_{\partial^*F}\cdot\xi)\circ\Psi^h \big) \dH.
\end{aligned}
\end{equation}
Next, we split the relative energy density into two parts. 
We define $\eta(x) := \max\{1 - (1 {-} |x|)^2, 0\}$;
in particular, for $\dist(x,\partial B_1)<1$ it holds $\eta(x) = 1 -  s^2(x)$, where $s:=s_{\partial B_1}$ denotes the signed distance to~$\partial B_1$
such that $\nabla s(x) = - \frac{x}{|x|}$ is the inward pointing normal
vector field along~$\partial B_1$. Recalling the definition~\eqref{eq:quantCalibration2}
of the quantitative calibration~$\xi$, we then decompose by adding zero
\begin{equation}
\begin{aligned}
\label{eq:FugledeAux4b}
1 - (\n_{\partial^*F}\cdot\xi)\circ\Psi^h
&= \big(1 - \eta\circ\Psi^h\big) + \eta\circ\Psi^h \big(1 - (\n_{\partial^*F}\circ\Psi^h)\cdot\n_{\partial B_1}\big) 
\\&
= h^2 + \eta\circ\Psi^h\big(1 - (\n_{\partial^*F}\circ\Psi^h)\cdot\n_{\partial B_1}\big).
\end{aligned}
\end{equation}
Taylor expanding the two resulting contributions in~\eqref{eq:FugledeAux3b} therefore yields:
For all $\gamma \in (0,1)$ there exists $\eps \ll_\gamma 1$ such that if $\|h\|_{W^{1,\infty}(\partial B_1)} \leq \eps$,
then
\begin{equation}
\begin{aligned}
\label{eq:FugledeAux5}
(1 {-} \gamma) \int_{\partial B_1} h^2 + \frac{1}{2}|\nabla_{\partial B_1} h|^2 \dH
&\leq E_{rel}[F|\xi].
\end{aligned}
\end{equation}

For the expansion of the remaining volume term, we first record the following computation
on $\{x\colon\dist(x,\partial B_1)<1\}$:
\begin{equation}
\label{eq:VolumeIntegrand}
\begin{aligned}
\nabla\cdot\xi + (d{-}1) &= -2s - \Big(1 {-} s^2\Big)\frac{d{-}1}{1-s}
+ (d{-}1)
= -(d{+}1)s.
\end{aligned}
\end{equation}
We thus obtain from an application of the coarea formula
\begin{equation}
\label{eq:FugledeAux6}
\begin{aligned}
- &\int_{\Rd} (\chi_{F} {-} \chi_{B_1}) \big(\nabla\cdot\xi + (d{-}1)\big) \dx
\\&
= -\int_{\partial B_1} \int_{0}^{h} (1{-}s)^{d-1}(d{+}1)s \,\mathrm{d}s\mathrm{d}\mathcal{H}^{d-1}.
\end{aligned}
\end{equation}
Hence, a Taylor expansion entails: For all $\gamma \in (0,1)$ there exists $\eps \ll_\gamma 1$ such 
that if $\|h\|_{W^{1,\infty}(\partial B_1)} \leq \eps$, then
\begin{equation}
\begin{aligned}
\label{eq:FugledeAux7}
-(1 {+} \gamma) \int_{\partial B_1} \frac{1}{2}(d{+}1) h^2 \dH
&\leq -\int_{\Rd} (\chi_{F} {-} \chi_{B_1}) \big(\nabla\cdot\xi + (d{-}1)\big) \dx. 
\end{aligned}
\end{equation}

With the expansions~\eqref{eq:FugledeAux5} and~\eqref{eq:FugledeAux7} in place,
we proceed by decomposing into spherical harmonics. Denoting by $\pi_k f$ the orthogonal projection
of $f \in L^2(\partial B_1)$ onto the subspace of spherical harmonics of degree~$k$, we obtain
\begin{equation}
\begin{aligned}
\label{eq:FugledeAux8}
\int_{\partial B_1} h^2 + \frac{1}{2}|\nabla_{\partial B_1} h|^2 \dH
= \sum_{k=0}^\infty \Big(1 + \frac{1}{2}k\big(k + (d{-}2)\big)\Big) \|\pi_k h\|^2_{L^2(\partial B_1)}
\end{aligned}
\end{equation}
as well as
\begin{equation}
\begin{aligned}
\label{eq:FugledeAux9}
- \int_{\partial B_1} \frac{1}{2}(d{+}1) h^2 \,d\mathcal{H}^{d-1}
= - \sum_{k=0}^\infty \frac{1}{2}(d{+}1) \|\pi_k h\|^2_{L^2(\partial B_1)}.
\end{aligned}
\end{equation}

Comparing coefficients, it follows that, as it should be, we have a spectral gap for modes $k \geq 2$,
whereas $k=0$ is unstable and $k=1$ is, to leading order, indeterminate. More precisely, for the stable modes it holds
\begin{equation}
\begin{aligned}
\label{eq:sharpSpectralGapStableModes}
&\sum_{k = 2}^\infty \frac{1}{2}(d{+}1) \|\pi_k h\|^2_{L^2(\partial B_1)}
\\&~~~
= \frac{1}{2} \sum_{k = 2}^\infty \frac{d{+}1}{1 + \frac{1}{2}k\big(k + (d{-}2)\big)} 
\Big(1 + \frac{1}{2}k\big(k + (d{-}2)\big)\Big) \|\pi_k h\|^2_{L^2(\partial B_1)}
\\&~~~
\leq \frac{1}{2} \sum_{k=2}^\infty \Big(1 + \frac{1}{2}k\big(k + (d{-}2)\big)\Big) \|\pi_k h\|^2_{L^2(\partial B_1)}
\end{aligned}
\end{equation}
since $C_k := \frac{d{+}1}{1 + \frac{1}{2}k(k + (d{-}2))} \leq C_2 = 1$ for $k \geq 2$.

The unstable and indeterminate modes may be corrected based on the last right hand side term of~\eqref{eq:relativeEnergyEquality}.
To be precise, fix constants $c_0 \in (0,\infty)$ as well as $c_1 \in (0,\infty)$ and choose
accordingly 
\begin{align}
\label{eq:choiceOfShifts}
a = c_0 \int_{\partial B_1} h \dH \in \Rd[] \quad\text{ and }\quad
\vec{b} = c_1\int_{\partial B_1} xh \dH \in \Rd
\end{align}
 such that
\begin{equation}
\begin{aligned}
\label{eq:FugledeAux10}
\int_{\partial B_1} a h + (\vec{b}\cdot x) h \dH
= c_0 \|\pi_0 h\|^2_{L^2(\partial B_1)} + c_1 \|\pi_1 h\|^2_{L^2(\partial B_1)}.
\end{aligned}
\end{equation}
A Taylor expansion then gives: For all $\gamma \in (0,1)$ there exists $\eps \ll_{\gamma,c_0,c_1} 1$ such 
that if $\|h\|_{W^{1,\infty}(\partial B_1)} \leq \eps$, then
\begin{equation}
\begin{aligned}
\label{eq:FugledeAux11}
&c_0 \|\pi_0 h\|^2_{L^2(\partial B_1)} + c_1 \|\pi_1 h\|^2_{L^2(\partial B_1)}
- \gamma E_{rel}[F|\xi]
\leq -\int_{\Rd} (\chi_F - \chi_{B_1}) L_{a,\vec{b}} \dx. 
\end{aligned}
\end{equation}

Inserting all this information into~\eqref{eq:relativeEnergyEquality}, we deduce that
for all $\gamma \in (0,1)$ there exists $\eps \ll_{\gamma,c_0,c_1} 1$ such 
that if $\|h\|_{W^{1,\infty}(\partial B_1)} \leq \eps$, then
\begin{equation}
\begin{aligned}
\label{eq:FugledeAux12}
\mathcal{H}^{d-1}(\partial^* F) - \mathcal{H}^{d-1}(\partial B_1) 
&\geq (1 {-} \gamma) E_{rel}[F|\xi]
\\&~~~
- \sum_{k\in\{0,1\}} \big((d{+}1)-c_k\big) \|\pi_kh\|_{L^2(\partial B_1)}^2
\\&~~~
- \frac{1}{2}(1{+}\gamma) \sum_{k=2}^\infty \Big(1 + \frac{1}{2}k\big(k + (d{-}2)\big)\Big) \|\pi_k h\|^2_{L^2(\partial B_1)}.
\end{aligned}
\end{equation}
A suitable choice of the constants $c_0 \in (0,\infty)$ and $c_1 \in (0,\infty)$ therefore allows to infer that
\begin{equation}
\begin{aligned}
\label{eq:FugledeAux13}
\mathcal{H}^{d-1}(\partial^* F) - \mathcal{H}^{d-1}(\partial B_1) 
&\geq (1 {-} \gamma) E_{rel}[F|\xi]
\\&~~~
- \frac{1}{2}(1{+}\gamma) \sum_{k=2}^\infty \Big(1 + \frac{1}{2}k\big(k + (d{-}2)\big)\Big) \|\pi_k h\|^2_{L^2(\partial B_1)}
\\&~~~
\geq (1 {-} \gamma) E_{rel}[F|\xi] - \frac{1}{2}\frac{1{+}\gamma}{1{-}\gamma} E_{rel}[F|\xi].
\end{aligned}
\end{equation}
This, however, entails~\eqref{eq:FugledeAux2b}.

\begin{remark}
\label{rem:proofClassicalFuglede}
We note that only the lower bound~\eqref{eq:FugledeAux5} actually required the assumption $\|h\|_{W^{1,\infty}(\partial B_1)} \ll 1$.
The remaining arguments are even valid under the less restrictive assumption $\|h\|_{L^\infty(\partial B_1)} \ll 1$.
\end{remark}

\subsection{Outline of strategy} 
The proof of our main contribution, Theorem~\ref{theo:pertRegime}, is based on
two explicit ``coarse-graining'' steps for the geometry of the competitor~$F$,
both being compatible with the structure of the relative energy equality~\eqref{eq:relativeEnergyEquality}
and the isoperimetric inequality~\eqref{eq:quantIsoperInequ2} we aim to derive from it.
\begin{itemize}
\item \textit{Step~1 (``tubular neighborhood averaging''):} We introduce a height function 
$h\colon \partial B_1 \to \mathbb{R}$ by integrating the symmetric difference $F\Delta B_1$
along normal rays from~$\partial B_1$. This height function serves as a $BV$ graph approximation~$F_h$
of the competitor~$F$. Within the relative energy inequality~\eqref{eq:relativeEnergyEquality}
we may pass from~$F$ to~$F_h$ since we prove in Lemma~\ref{lem:BVgraphApprox}
\begin{align}
E_{rel}[F|\xi] \ll 1 \quad\Rightarrow\quad 
\bigg|\int_{F\Delta F_h} \nabla\cdot\xi + (d{-}1) \dx\bigg|
\ll E_{rel}[F|\xi]. 
\end{align}
Furthermore, we establish in Lemma~\ref{lem:L2estimatesBVgraphApprox} a careful $BV$ analogue
of the second lower bound of~\eqref{eq:FugledeAux2b}, i.e., that relative energy controls suitably
height $h$ and slope $D^{tan}h$. Roughly speaking, we show that relative energy provides $L^2$ control in
terms of height, $L^2$ control in terms of slope for parts of~$\partial^*F$ already being
in a ``Fuglede regime'', and total variation control in terms of slope for all other parts of~$\partial^*F$.
\item \textit{Step~2 (``smoothing by mollification''):} To enter the spectral analysis by spherical harmonics,
we mollify the $BV$ graph approximation~$F_h$ at a scale~$\lambda$ directly linked to the relative energy~$E_{rel}[F|\xi]$
and obtain a smooth graph approximation~$F_{h_\lambda}$. The associated smooth height function~$h_\lambda$
essentially satisfies the second lower bound of~\eqref{eq:FugledeAux2b}
\begin{align}
E_{rel}[F|\xi] \ll 1 \quad\Rightarrow\quad 
\|h\|^2_{L^2(\partial B_1)} + \frac{1}{2}\|\nabla_{\partial B_1}h\|^2_{L^2(\partial B_1)}
\leq E_{rel}[F|\xi].
\end{align}
The precise statement is the content of Lemma~\ref{lem:L2estimatesSmoothGraphApprox}.
We may also again pass from~$F_h$ to $F_{h_\lambda}$
in the relative energy inequality~\eqref{eq:relativeEnergyEquality} by
showing in Lemma~\ref{lem:smoothGraphApprox}
\begin{align}
E_{rel}[F|\xi] \ll 1 \quad\Rightarrow\quad 
\bigg|\int_{F_h\Delta F_{h_\lambda}} \nabla\cdot\xi + (d{-}1) \dx\bigg|
\ll E_{rel}[F|\xi]. 
\end{align}
\end{itemize}

After these two coarse-graining steps, which themselves are derived by a careful
slicing argument for which we set the stage in Section~\ref{subsubsec:coercivity}, 
we may close the estimates by our interpretation
of Fuglede's argument from Section~\ref{subsec:classicalFuglede}. The above two
coarse-graining steps are in their essence not new and are, for a different
purpose in a different setting, in non-optimized form already present in~\cite{FischerHensel2020}
(see also~\cite{fischer2023} for some multi-phase versions of our slicing arguments).
In the present contribution, we are in contrast to~\cite{FischerHensel2020}, however, not allowed
to loose arbitrary constants in the estimates as our goal is to absorb the bulk term
from the relative energy equality into a fraction of the relative energy.
Within this coercivity argument, we also aim to identify the same spectral gap
as in the first lower bound of~\eqref{eq:FugledeAux2b}; now, of course, under the
substantially weaker smallness of excess assumption~\eqref{def:relativeEnergy}. These requirements
lead to the streamlined and optimized version of the geometry coarse-graining 
of the present contribution.  

\section{Proofs of main results}
	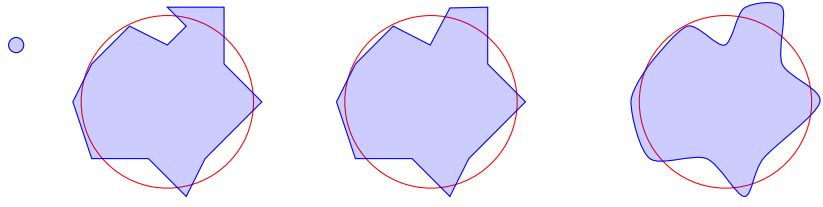
\begin{figure}
		\begin{subfigure}{0.3\textwidth}
			\centering
            \begin{tikzpicture}[scale=1]
              \coordinate (A) at (0.00, 0.00);
              \coordinate (B) at (0.00, 0.00);
              \coordinate (C) at (-1.00, 0.50);
              \coordinate (D) at (-0.50, 1.00);
              \coordinate (E) at (0.00, 0.75);
              \coordinate (F) at (0.25, 1.00);
              \coordinate (G) at (0.00, 1.25);
              \coordinate (H) at (0.75, 1.25);
              \coordinate (I) at (0.75, 0.50);
              \coordinate (J) at (1.25, 0.00);
              \coordinate (K) at (0.50, -0.75);
              \coordinate (L) at (0.25, -1.25);
              \coordinate (M) at (-0.25, -0.75);
              \coordinate (N) at (-1.00, -0.75);
              \coordinate (O) at (-1.25, 0.00);
              \coordinate (P) at (1.25, 0.75);
              \coordinate (Q) at (1.25, 0.75);
              \coordinate (R) at (-2.00, 0.75);
              \coordinate (S) at (-2.00, 0.75);


              \draw[red] (A) circle (1.14);

              \draw[blue] (C) -- (D) -- (E) -- (F) -- (G) -- (H) -- (I) -- (J) -- (K) -- (L) -- (M) -- (N) -- (O) -- cycle;
              \draw[fill,blue, opacity=0.2] (C) -- (D) -- (E) -- (F) -- (G) -- (H) -- (I) -- (J) -- (K) -- (L) -- (M) -- (N) -- (O) -- cycle;

              \draw[blue] (R) circle (0.10);
              \draw[fill,blue, opacity=0.2] (R) circle (0.10);
            \end{tikzpicture}
          \end{subfigure}
		  \begin{subfigure}{0.3\textwidth}
			\centering
              \begin{tikzpicture}[scale=1]
                \coordinate (A) at (0.00, 0.00);
                \coordinate (B) at (0.00, 0.00);
                \coordinate (C) at (-1.00, 0.50);
                \coordinate (D) at (-0.50, 1.00);
                \coordinate (E) at (-0.01, 0.75);
                \coordinate (F) at (0.25, 1.24);
                \coordinate (H) at (0.75, 1.25);
                \coordinate (I) at (0.75, 0.50);
                \coordinate (J) at (1.25, 0.00);
                \coordinate (K) at (0.50, -0.75);
                \coordinate (L) at (0.25, -1.25);
                \coordinate (M) at (-0.25, -0.75);
                \coordinate (N) at (-1.00, -0.75);
                \coordinate (O) at (-1.25, 0.00);
                \coordinate (P) at (1.25, 0.75);
                \coordinate (Q) at (1.25, 0.75);
                \coordinate (R) at (-2.00, 0.75);
                \coordinate (S) at (-2.00, 0.75);

                \draw[red] (A) circle (1.14);

                \draw[blue] (C) -- (D) -- (E) -- (F) -- (H) -- (I) -- (J) -- (K) -- (L) -- (M) -- (N) -- (O) -- cycle;

                \draw[fill,blue, opacity=0.2] (C) -- (D) -- (E) -- (F) -- (H) -- (I) -- (J) -- (K) -- (L) -- (M) -- (N) -- (O) -- cycle;
              \end{tikzpicture}
            \end{subfigure}
			\begin{subfigure}{0.3\textwidth}
				\centering
				\begin{tikzpicture}[scale=1]
                \coordinate (A) at (0.00, 0.00);
                \coordinate (B) at (0.00, 0.00);
                \coordinate (C) at (-1.00, 0.50);
                \coordinate (D) at (-0.50, 1.00);
                \coordinate (E) at (-0.01, 0.75);
                \coordinate (F) at (0.25, 1.24);
                \coordinate (H) at (0.75, 1.25);
                \coordinate (I) at (0.75, 0.50);
                \coordinate (J) at (1.25, 0.00);
                \coordinate (K) at (0.50, -0.75);
                \coordinate (L) at (0.25, -1.25);
                \coordinate (M) at (-0.25, -0.75);
                \coordinate (N) at (-1.00, -0.75);
                \coordinate (O) at (-1.25, 0.00);
                \coordinate (P) at (1.25, 0.75);
                \coordinate (Q) at (1.25, 0.75);
                \coordinate (R) at (-2.00, 0.75);
                \coordinate (S) at (-2.00, 0.75);

                \draw[red] (A) circle (1.14);

                \draw [blue] plot [smooth, tension=0.5] coordinates {(C) (D) (E) (F) (H) (I) (J) (K) (L) (M) (N) (O) (C)};
                \draw[fill,blue, opacity=0.2] plot [smooth, tension=0.5] coordinates {(C) (D) (E) (F) (H) (I) (J) (K) (L) (M) (N) (O) (C)};

              \end{tikzpicture}
            \end{subfigure}
			\caption{Starting from a general set of finite perimeter (left), we first approximate it by a $BV$ graph over $\partial B_1$ (middle), and then by a smooth one (right). The key is to control the error in each step by the relative energy $E_{rel}[F|\xi]$.}
			\label{fig:graph_approximation}
    \end{figure}
The following subsections~\ref{subsubsec:coercivity}--\ref{subsec:proofTheoremPertRegime} 
provide a proof of Theorem~\ref{theo:pertRegime}, for which we fix throughout
a set of finite perimeter~$F \subset \Rd$ satisfying the constraints~\eqref{eq:volumeConstraint2}--\eqref{eq:smallRelEnergy},
and determine the constant~$\eps \in (0,1)$
in the course of the proof. 

\subsection{Coercivity of the relative energy}\label{subsubsec:coercivity}
It is our goal to specify the error control provided
by the relative energy. To this end, we distinguish between its
behaviour within and outside a tubular neighborhood around~$\partial B_1$.
We start with the former.

Let~$\delta \in (0,\frac{1}{4})$, define 
$\mathbb{T}_\pi := [0,\pi)^{d-2} \times [0,2\pi)$, and consider the map
\begin{equation}
\begin{aligned}
\label{eq:polarCoordinates}
f\colon (-\infty,1) \times \mathbb{T}_\pi &\to \Rd\setminus\{0\},
\\
(s,\vec{\theta}) &\mapsto f(s,\vec{\theta}),
\end{aligned}
\end{equation}
induced by polar coordinates, where the radial variable is replaced by the signed distance~$s$ to~$\partial B_1$.
The restriction of~$f$ to the interior of~$\mathbb{T}_\pi$ 
is a diffeomorphism onto~$\Rd\setminus\{x \in \Rd\colon x_d = 0,\, x_{d-1} \geq 0\}$.
By stability of sets of finite perimeter under 
diffeomorphic coordinate changes (cf.\ \cite[Section~17.1]{Maggi2012}) 
and $BV$ slicing theory (cf.\ \cite[Section~3.11]{AmbrosioFuscoPallara}),
we know that for $\mathcal{L}^{d{-}1}$-a.e.\ $\vec{\theta} \in \mathbb{T}_\pi$
\begin{itemize}
\item[$(\mathrm{S1})_{\vec{\theta}}$] 
the set $F_{\vec{\theta}} := \{s \in (-\delta,\delta) \colon f(s,\vec{\theta}) \in F\}$ is a set of finite
perimeter in~$(-\delta,\delta)$, i.e., a union of finitely many open and bounded intervals,
\item[$(\mathrm{S2})_{\vec{\theta}}$]  
$\chi_F(f(s,\vec{\theta}) ) = \chi^*_F(f(s,\vec{\theta}))$ for $\mathcal{L}^1$-a.e.\ $s \in (-\delta,\delta)$,
where $\chi^*_F$ denotes the precise representative of~$\chi_F$,
\item[$(\mathrm{S3})_{\vec{\theta}}$]   
$\partial^*F_{\vec{\theta}} = \{s \in (-\delta,\delta) \colon f(s,\vec{\theta}) \in \partial^*F\}$,
\item[$(\mathrm{S4})_{\vec{\theta}}$]
 $\n_{\partial^*F}(f(s,\vec{\theta})) \cdot \n_{\partial B_1}(f(0,\vec{\theta})) \neq 0$ for all $s \in \partial^*F_{\vec{\theta}}$,
\item[$(\mathrm{S5})_{\vec{\theta}}$]   
$\lim_{\tilde{s} \downarrow s} \chi^*_F(f(\tilde{s},\vec{\theta})) = 1$ as well as $\lim_{\tilde{s} \uparrow s} \chi^*_F(f(\tilde{s},\vec{\theta})) = 0$ 
if $\n_{\partial^*F}(f(s,\vec{\theta})) \cdot \n_{\partial B_1}(f(0,\vec{\theta})) > 0$, $s \in \partial^*F_{\vec{\theta}}$,
\item[$(\mathrm{S6})_{\vec{\theta}}$] 
$\lim_{\tilde{s} \downarrow s} \chi^*_F(f(\tilde{s},\vec{\theta})) = 0$ as well as $\lim_{\tilde{s} \uparrow s} \chi^*_F(f(\tilde{s},\vec{\theta})) = 1$
if $\n_{\partial^*F}(f(s,\vec{\theta})) \cdot \n_{\partial B_1}(f(0,\vec{\theta})) < 0$, $s \in \partial^*F_{\vec{\theta}}$.
\end{itemize}

Denoting the $\mathcal{L}^{d-1}$ null set in $\mathbb{T}_\pi$ for which the above does not apply by~$\mathcal{N}$, i.e.,
\begin{equation}
\begin{aligned}
\label{eq:nullSetSlicing}
\mathbb{T}_\pi \setminus \mathcal{N} = \{\vec{\theta}\in \mathbb{T}_\pi\colon (\mathrm{S1})_{\vec{\theta}},\ldots,(\mathrm{S6})_{\vec{\theta}} \text{ apply}\},
\end{aligned}
\end{equation}
we now decompose $\mathbb{T}_\pi \setminus \mathcal{N}$ in the form of
\begin{align}
\label{eq:decomp}
\mathbb{T}_\pi \setminus \mathcal{N} = \Theta_{\text{outside}} \cup \Theta_{\text{graph}} \cup  \Theta_{\text{slope}},
\end{align}
where 
\begin{align}
\label{eq:defThetaOutside}
\Theta_{\text{outside}} &:= \big\{\vec{\theta} \in \mathbb{T}_\pi \setminus \mathcal{N} \colon 
\#\partial^*F_{\vec{\theta}} = 0\big\},
\\
\label{eq:defThetaGraph}
\Theta_{\text{graph}} &:= \big\{\vec{\theta} \in \mathbb{T}_\pi \setminus \mathcal{N} \colon 
\partial^*F_{\vec{\theta}} = \{s\}, \n_{\partial^*F}(f(s,\vec{\theta})) \cdot \n_{\partial B_1}(f(0,\vec{\theta})) > 0\,\big\},
\end{align}
and
\begin{equation}
\begin{aligned}
\label{eq:defThetaSlope}
\Theta_{\text{slope}} &:= \big\{\vec{\theta} \in \mathbb{T}_\pi \setminus \mathcal{N} \colon 
\text{ either } \#\partial^*F_{\vec{\theta}} \geq 2 \text{ or } \partial^*F_{\vec{\theta}} = \{s\}
\\&~~~~~~~~~~~~~~~~~~~~~~~~
\text{ such that } \n_{\partial^*F}(f(s,\vec{\theta})) \cdot \n_{\partial B_1}(f(0,\vec{\theta})) < 0\,\big\}.
\end{aligned}
\end{equation}

Or in words, $\Theta_{\text{outside}}$ represents the set of those angles~$\vec{\theta}\in \mathbb{T}_\pi$ for which the interface of~$F$
does not appear in the slice~$F_{\vec{\theta}}$, $\Theta_{\text{graph}}$ represents the set of those angles~$\vec{\theta}\in \mathbb{T}_\pi$ for which the interface of~$F$ appears once and with correct orientation in the slice~$F_{\vec{\theta}}$,
and $\Theta_{\text{slope}}$ represents the set of those angles~$\vec{\theta}\in \mathbb{T}_\pi$ 
for which the interface of~$F$ appears
either with overlaps or exactly once but with wrong orientation in the slice~$F_{\vec{\theta}}$. 

It will also be crucial to identify the set of those angles~$\vec{\theta}\in \Theta_{\text{graph}}$ 
corresponding to perturbatively small slope of the height function,
i.e., the Fuglede regime:
\begin{align} 
\label{eq:defThetaFuglede}
\Theta_{\text{Fuglede}} := \Big\{\vec{\theta} \in \Theta_{\text{graph}}\colon
\n_{\partial^*F}(f(s,\vec{\theta})) \cdot \n_{\partial B_1}(f(0,\vec{\theta})) 
\geq \frac{1}{\sqrt{1{+}c^2}}\Big\},
\end{align}
where $c \in (0,1]$.

Define the relative energy density $e_{rel}(s,\vec{\theta}) := (1 - \n_{\partial^*F}\cdot\xi)(f(s,\vec{\theta}))$,
$(s,\vec{\theta}) \in (-\delta,\delta){\times} \mathbb{T}_\pi$, as well as $\eta(s) := 1-s^2$,
$s \in (-\delta,\delta)$. In particular, 
\begin{align}
\label{eq:relEnergyDensity}
e_{rel}(s,\vec{\theta}) = 1 - \eta(s)\n_{\partial^*F}(f(s,\vec{\theta}))\cdot \n_{\partial B_1}(f(0,\vec{\theta})),
\quad (s,\vec{\theta}) \in (-\delta,\delta){\times} \mathbb{T}_\pi.
\end{align} 
Defining $\mathcal{N}^\delta := \{f(s,\vec{\theta})\colon \vec{\theta} \in \mathcal{N}, \,s \in (-\delta,\delta)\}$,
we then obtain from the area formula
\begin{equation}
\label{eq:errorControlInSlices}
\begin{aligned}
&\int_{\partial^*F \cap (\{\dist(\cdot,\partial B_1) < \delta\} \setminus \mathcal{N}^\delta)} 
1 - \n_{\partial^*F} \cdot \xi \dH
\\&~~~
=\int_{\mathbb{T}_\pi \setminus \mathcal{N}}
\sum_{s \in \partial^*F_{\vec{\theta}}} \frac{(1 {-} s)^{d{-}1}}{|
\n_{\partial^*F}(f(s,\vec{\theta}))\cdot \n_{\partial B_1}(f(0,\vec{\theta}))|}
e_{rel}(s,\vec{\theta})  \dvectheta.
\end{aligned}
\end{equation}
Since
\begin{equation}
\begin{aligned}
e_{rel}(s,\vec{\theta}) &= \big(1 {-} \eta(s)\big) 
+ \eta(s)\big(1 {-} \n_{\partial^*F}(f(s,\vec{\theta}))\cdot \n_{\partial B_1}(f(0,\vec{\theta}))\big)
\\
&= s^2 + (1 {-} s^2)\big(1 {-} \n_{\partial^*F}(f(s,\vec{\theta}))\cdot \n_{\partial B_1}(f(0,\vec{\theta}))\big),
\end{aligned}
\end{equation}
we get
\begin{equation}
\label{eq:errorControlSlice}
\begin{aligned}
&\int_{\partial^*F \cap (\{\dist(\cdot,\partial B_1) < \delta\} \setminus \mathcal{N}^\delta)} 
1 - \n_{\partial^*F} \cdot \xi \dH
\\&~~~
\geq \int_{\mathbb{T}_\pi \setminus \mathcal{N}}
\sum_{s \in \partial^*F_{\vec{\theta}}} \frac{(1 {-} s)^{d{-}1}}{|
\n_{\partial^*F}(f(s,\vec{\theta}))\cdot \n_{\partial B_1}(f(0,\vec{\theta}))|}
\big(e_{height}(s,\vec{\theta}) + e_{slope}(s,\vec{\theta})\big)  \dvectheta,
\end{aligned}
\end{equation}
where in the last line we abbreviated
\begin{align}
\label{def:heightControl}
e_{height}(s,\vec{\theta}) &:= s^2,
\\
\label{def:orientControl}
e_{slope}(s,\vec{\theta}) &:= (1 {-} \delta^2)
\big(1 {-} \n_{\partial^*F}(f(s,\vec{\theta}))\cdot \n_{\partial B_1}(f(0,\vec{\theta}))\big).
\end{align}
We also control the length of vertical components of the competitor interface by means of
\begin{equation}
\begin{aligned}
\label{eq:errorControlVertical}
&\int_{\partial^*F \cap \{\dist(\cdot,\partial B_1) < \delta\} \cap \mathcal{N}^\delta} 
1 - \n_{\partial^*F} \cdot \xi \dH
\\&~~~
\geq \mathcal{H}^{d-1}\bigg(\partial^*F \cap \{\dist(\cdot,\partial B_1) < \delta\} \cap \mathcal{N}^\delta
\cap \Big\{x \colon \n_{\partial^*F}(x)\cdot \n_{\partial B_1}\Big(\frac{x}{|x|}\Big)=0\Big\}\bigg).
\end{aligned}
\end{equation}

For $x \in \Rd$ with $\dist(x,\partial B_1)\geq \delta$, we simply note that
$(1 {-} \n_{\partial^*F}\cdot\xi)(x) \geq \delta^2$. Hence, outside the $\delta$-tubular neighborhood we 
control the length of the interface of the competitor in the sense of
\begin{align}
\label{eq:errorControlOutside}
\int_{\partial^*F \cap \{\dist(\cdot,\partial B_1) \geq \delta\}}  
1 - \n_{\partial^*F} \cdot \xi \dH
\geq \delta^2\mathcal{H}^{d-1}\big(\partial^*F \cap \{\dist(\cdot,\partial B_1) \geq \delta\}\big).
\end{align}

\subsection{From general competitor to its $BV$ graph approximation}
Let $\delta \in (0,\frac{1}{4})$.
We define 
a height function~$h\colon \mathbb{T}_\pi \to [-\delta,\delta]$
\begin{align}
\label{def:height}
h(\vec{\theta}) := -\int_{-\delta}^{\delta} (\chi_F {-} \chi_{B_1})(f(s,\vec{\theta})) \ds.
\end{align}
In what follows, we will always identify~$h$ with its periodic extension
from $\mathbb{T}_\pi = [0,\pi)^{d-2}{\times}[0,2\pi)$
to~$\mathbb{R}^{d-1}$. The point of this height function is that it provides a $BV$ graph approximation~$\chi_h$ to~$\chi_F$
up to an admissible error in the relative energy:
\begin{align}
\label{eq:BVgraphApprox}
\chi_h(f(s,\vec{\theta})) := \chi_{B_1}(f(s,\vec{\theta})) 
												- \chi_{\{(s,\vec{\theta})\colon 0 \leq s \leq h(\vec{\theta})\}}(s,\vec{\theta})
												+ \chi_{\{(s,\vec{\theta})\colon h(\vec{\theta}) \leq s \leq 0\}}(s,\vec{\theta}).
\end{align}
Indeed:
\begin{lemma}
\label{lem:BVgraphApprox}
For every $C' \in (1,\infty)$
one may choose $\delta \ll 1$ and $\eps \ll_\delta 1$ such that
\begin{align}
-\int_{\Rd} (\chi_F - \chi_h) (\nabla\cdot\xi + (d{-}1)) \dx 
\geq - \frac{1}{C'} E_{rel}[F|\xi].
\end{align}
\end{lemma}

\begin{proof}[Proof of Lemma~\ref{lem:BVgraphApprox}]
By the change of variables $(-\infty,1) \times \mathbb{T}_\pi \ni (s,\vec{\theta}) \mapsto x := 
f(s,\vec{\theta})$ with associated Jacobian determinant~$\mathcal{J}_f(s,\vec{\theta})$,
we may split the task into three parts:
\begin{align}
\int_{\Rd} (\chi_F - \chi_h) (\nabla\cdot\xi + (d{-}1)) \dx =: I + I\!I + I\!I\!I,
\end{align}
where
\begin{align}
I &:= \int_{\mathbb{T}_\pi} \int_{-\delta}^{\delta} 
\big(\chi_F {-} \chi_h\big)(f(s,\vec{\theta})) \big(\nabla\cdot\xi {+} (d{-}1)\big)(f(s,\vec{\theta})) \mathcal{J}_f(s,\vec{\theta})
\ds \dvectheta,
\\
I\!I &:= \int_{\mathbb{T}_\pi} \int_{\delta}^{1}
\big(\chi_F {-} 1\big)(f(s,\vec{\theta})) \big(\nabla\cdot\xi {+} (d{-}1)\big)(f(s,\vec{\theta})) \mathcal{J}_f(s,\vec{\theta})
\ds \dvectheta, 
\\
I\!I\!I &:= \int_{\mathbb{T}_\pi} \int_{-\infty}^{-\delta}
\chi_F(f(s,\vec{\theta})) \big(\nabla\cdot\xi {+} (d{-}1)\big)(f(s,\vec{\theta})) \mathcal{J}_f(s,\vec{\theta})
\ds \dvectheta. 
\end{align}

We start with the contribution from~$I\!I\!I$
and note that, since~$\xi$ is compactly supported, $|I\!I\!I| \lesssim \mathcal{L}^d((\Rd\setminus B_{1+\delta}) \cap F)$.
Due to the volume constraint~\eqref{eq:volumeConstraint2},
$\mathcal{L}^d((\Rd\setminus B_{1+\delta}) \cap F) \leq \omega_d$ and thus by the relative isoperimetric inequality in~$\Rd\setminus B_{1+\delta}$
it follows $|I\!I\!I| \lesssim \mathcal{H}^{d{-}1}((\Rd\setminus B_{1+\delta}) \cap \partial^*F)^\frac{d}{d-1}$. However, 
by the error control~\eqref{eq:errorControlOutside},
$\mathcal{H}^{d-1}((\Rd\setminus B_{1+\delta}) \cap \partial^*F) \leq \frac{1}{\delta^2}E_{rel}[F|\xi]$
so that the smallness of the relative energy~\eqref{eq:smallRelEnergy} implies for $\eps \ll_\delta 1$
\begin{align}
|I\!I\!I| \leq \frac{1}{3C'}E_{rel}[F|\xi].
\end{align} 

The argument for the required estimate of~$I\!I$ is similar. However, in order to avoid a degenerating constant from
the relative isoperimetric inequality in~$B_{1-\delta}$, we actually perform a case distinction. To this end,
we first consider the case that $\mathcal{L}^d\big(B_{1-\delta}\setminus F\big) \leq \frac{1}{2}(1{-}\delta)^d\omega_d$. 
Arguing as for~$I\!I\!I$, but now based on the relative isoperimetric inequality in~$B_{1-\delta}$, shows
$|I\!I| \leq \frac{1}{3C'}E_{rel}[F|\xi]$. In the other case, recalling the volume constraint~\eqref{eq:volumeConstraint2}
and noting that $\mathcal{L}^d(B_{1+\delta}\setminus B_{1-\delta}) = O(\delta)\omega_d$, it follows
$\mathcal{L}^d((\Rd\setminus B_{1+\delta}) \cap F) \geq \omega_d - O(\delta)\omega_d - \frac{1}{2}(1{-}\delta)^2\omega_d$, 
which may be bounded from below by $\frac{\omega_d}{4}$
provided $\delta\ll 1$. Hence, provided $\delta \ll 1$, we may estimate $\mathcal{L}^d(B_{1-\delta}\setminus F) \leq (1{-}\delta)^2\omega_d
= 4(1{-}\delta)^2\frac{\omega_d}{4}\leq 4(1-\delta)^2 \mathcal{L}^d((\Rd\setminus B_{1+\delta}) \cap F)$.
The argument for~$I\!I\!I$ therefore enables us to conclude the second case, so that in both cases
for $\delta \ll 1$ and $\eps \ll_\delta 1$
\begin{align}
|I\!I| \leq \frac{1}{C'}E_{rel}[F|\xi].
\end{align} 

It remains to estimate the contribution from~$I$. First note that
for $\vec{\theta} \in \Theta_{\text{outside}} \cup \Theta_{\text{graph}}$
it holds $\big(\chi_F {-} \chi_h\big)(f(s,\vec{\theta})) = 0$
for a.e.\ $s \in (-\delta,\delta)$.
Hence, we only have to provide
an estimate for $\vec{\theta} \in \Theta_{\text{slope}}$. 
A crude estimate then gives
\begin{align}
|I| \lesssim \int_{\Theta_{\text{slope}}} \int_{-\delta}^{\delta} (1{-}s)^{d-1} \ds \dvectheta
= \mathcal{L}^{d-1}(\Theta_{\text{slope}}) O(\delta). 
\end{align}
Next, we claim
\begin{align}
\label{eq:controlAmountOfOVerlaps}
\mathcal{L}^{d-1}(\Theta_{\text{slope}}) \lesssim E_{rel}[F|\xi].
\end{align}
Indeed, for each $\vec{\theta} \in \Theta_{\text{slope}}$
we may find $s_{\vec{\theta}} \in \partial^*F_{\vec{\theta}}$ such that
$\n_{\partial^*F}(f(s_{\vec{\theta}},\vec{\theta})) \cdot \n_{\partial B_1}(f(0,\vec{\theta})) < 0$;
in particular, $e_{slope}(s_{\vec{\theta}},\vec{\theta}) \geq 1 - \delta^2$,
see~\eqref{def:orientControl}. Hence,
\begin{equation}
\begin{aligned}
&\mathcal{L}^{d-1}(\Theta_{\text{slope}}) 
\\&~~~
\leq (1 {+} O(\delta)) \int_{\Theta_{\text{slope}}} 
\frac{(1 {-} s_{\vec{\theta}})^{d{-}1}}{|
\n_{\partial^*F}(f(s_{\vec{\theta}},\vec{\theta})) \cdot \n_{\partial B_1}(f(0,\vec{\theta}))|}
e_{slope}(s_{\vec{\theta}},\vec{\theta}) \dvectheta
\end{aligned}
\end{equation}
so that~\eqref{eq:controlAmountOfOVerlaps} follows from~\eqref{eq:errorControlSlice}.
In summary, we obtain for $\delta \ll 1$
\begin{align}
\label{eq:aux1}
|I| \leq \frac{1}{3C'}E_{rel}[F|\xi]. 
\end{align}  
This concludes the proof.
\end{proof}

We will also require the following properties:
\begin{lemma}
\label{lem:L2estimatesBVgraphApprox}
The height function~$h$ defined by~\eqref{def:height} 
satisfies $h \in BV(\mathbb{T}_\pi)$, and we denote
by $Dh = \nabla h\dvectheta + D^sh$ the associated Radon--Nikodym decomposition
of the distributional derivative $Dh$ (recall here that we always identify~$h$
with its periodic extension from $\mathbb{T}_\pi = [0,\pi)^{d-2}{\times}[0,2\pi)$
to~$\mathbb{R}^{d-1}$). 

Furthermore, for every $C'\in(1,\infty)$ one may choose $\delta \ll 1$ and $\eps \ll_\delta 1$
such that we obtain 
\begin{itemize}
\item[i)] a global $L^2$ estimate
for the height of the $BV$ graph approximation (recall the notation from~\ref{subsubsec:coercivity})
\begin{equation}
\begin{aligned}
\label{eq:L2estimateHeight}
&\int_{\mathbb{T}_\pi} h^2 \dvectheta 
\\&~~~
\leq \big(1{+}O(\delta)\big)  \int_{\mathbb{T}_\pi\setminus \mathcal{N}} 
\sum_{s \in \partial^*F_{\vec{\theta}}}  \frac{(1 {-} s)^{d{-}1}}{|
\n_{\partial^*F}(f(s,\vec{\theta})) \cdot \n_{\partial B_1}(f(0,\vec{\theta}))|}
e_{height}(s,\vec{\theta})  \dvectheta
\\&~~~~~~
+ \frac{1}{C'} E_{rel}[F|\xi],
\end{aligned}
\end{equation}
\item[ii)]
in the graph regime with perturbatively small slope (i.e., in the Fuglede regime~\eqref{eq:defThetaFuglede})
an $L^2$ estimate for the absolutely continuous part of~$Dh$ 
\begin{equation}
\begin{aligned}
\label{eq:L2estimateDerivativeHeightL2}
&\int_{\mathbb{T}_\pi \cap \Theta_{\emph{\text{Fuglede}}}} |\nabla h|^2 \dvectheta
\\&~~~
\leq \big(2\sqrt{1{+}c^2} {+} O(\delta)\big) 
\\&~~~~~~~~~
\times 
\int_{(\mathbb{T}_\pi\setminus \mathcal{N}) \cap \Theta_{\emph{\text{Fuglede}}}}  
\sum_{s \in \partial^*F_{\vec{\theta}}}  \frac{(1 {-} s)^{d{-}1}}{|
\n_{\partial^*F}(f(s,\vec{\theta})) \cdot \n_{\partial B_1}(f(0,\vec{\theta}))|}
e_{slope}(s,\vec{\theta}) \dvectheta,
\end{aligned}
\end{equation}
\item[iii)]
in all other regimes at least a total variation estimate for the absolutely continuous part of~$Dh$ 
\begin{equation}
\begin{aligned}
\label{eq:L2estimateDerivativeHeightL1}
&\int_{\mathbb{T}_\pi \setminus \Theta_{\emph{\text{Fuglede}}}} |\nabla h| \dvectheta
\\&~~~
\leq \bigg(\max\Big\{3,\Big(1{-}\frac{1}{\sqrt{1{+}c^2}}\Big)^{-1}\Big\} {+} O(\delta)\bigg) 
\\&~~~~~~~~~
\times 
\int_{\mathbb{T}_\pi\setminus (\mathcal{N}  \cup \Theta_{\emph{\text{Fuglede}}})}  
\sum_{s \in \partial^*F_{\vec{\theta}}}  \frac{(1 {-} s)^{d{-}1}}{|
\n_{\partial^*F}(f(s,\vec{\theta})) \cdot \n_{\partial B_1}(f(0,\vec{\theta}))|}
e_{slope}(s,\vec{\theta}) \dvectheta,
\end{aligned}
\end{equation}
\item[iv)]
and finally a total variation estimate for the singular part of~$Dh$ 
\begin{equation}
\begin{aligned}
\label{eq:L2estimateDerivativeHeight2}
&\|D^sh\|_{TV(\mathbb{T}_\pi)} 
\\&~~~
\leq \mathcal{H}^{d-1}\bigg(\partial^*F \cap \{\dist(\cdot,\partial B_1) < \delta\} \cap \mathcal{N}^\delta
\cap \Big\{x \colon \n_{\partial^*F}(x)\cdot \n_{\partial B_1}\Big(\frac{x}{|x|}\Big)=0\Big\}\bigg).
\end{aligned}
\end{equation}
\end{itemize}
\end{lemma}

\begin{proof}[Proof of Lemma~\ref{lem:L2estimatesBVgraphApprox}]
We split the proof into three steps. 

\textit{Step~1 (Proof of $h \in BV(\mathbb{T}_\pi)$).} We first introduce an auxiliary
height function $\tilde h\colon \partial B_1 \to [-\delta,\delta]$
by means of $\tilde h(f(s,\vec{\theta})) := h(\vec{\theta})$, $\vec{\theta} \in \mathbb{T}_\pi$.
Let $\eta \in C^\infty(\partial B_1;\Rd)$ such that
$\eta(x) \in Tan_{x}\partial B_1$ for all $x \in \partial B_1$.
Then, with the convention $\tilde\eta(x) := \eta(x{+}s(x)\frac{x}{|x|})$, $x \in \{\dist(\cdot,\partial B_1) \leq \delta\}$,
we first obtain by the chain rule and $\nabla^2s(x) = - \frac{1}{1-s(x)}(\mathrm{Id}-\nabla s(x)\otimes\nabla s(x))$
that
\begin{equation}
\begin{aligned}
\nabla\tilde\eta(x) &= (\nabla^{tan}\eta)\Big(x{+}s(x)\frac{x}{|x|}\Big)
\bigg(\mathrm{Id} - \nabla s(x) \otimes \nabla s(x) - s(x)\nabla^2s(x)\bigg)
\\&
= (\nabla^{tan}\eta)\Big(x{+}s(x)\frac{x}{|x|}\Big)
\frac{1}{1-s(x)}\bigg(\mathrm{Id} - \nabla s(x) \otimes \nabla s(x)\bigg).
\end{aligned}
\end{equation}
As a consequence, we may compute based on an integration by parts
along~$\partial B_1$, the definition of the height function~$\tilde h$,
an integration by parts within the tubular neighborhood $x \in \{\dist(\cdot,\partial B_1) < \delta\}$,
and the fact that $\tilde\eta(x) \cdot \nabla s(x) =
\tilde\eta(x) \cdot \vec{n}_{\partial B_1}(\frac{x}{|x|}) = - \tilde\eta(x) \cdot\frac{x}{|x|} = 0$,
\begin{equation}
\begin{aligned}
\langle D^{tan}\tilde{h}, \eta \rangle 
&= - \int_{\partial B_1} \tilde{h} \nabla^{tan}\cdot \eta \dS(x)
\\
&= \int_{\partial B_1} (\nabla^{tan}\cdot\eta)(x) \int_{-\delta}^{\delta} 
(\chi_F{-}\chi_{B_1})\Big(x {-} s \frac{x}{|x|}\Big) \ds \dS(x) 
\\
&= \int_{\partial B_1} \int_{-\delta}^{\delta} 
\big((\nabla\cdot\tilde\eta)(\chi_F{-}\chi_{B_1})\big)
\Big(x {-} s \frac{x}{|x|}\Big) (1 {-} s)\ds \dS(x)
\\&
= \int_{\{x\colon\dist(x,\partial B_1) < \delta\}} (\nabla\cdot\tilde\eta) (\chi_F{-}\chi_{B_1})
\frac{1}{(1{-}s)^{d-2}} \dx
\\&
= - \int_{\partial^*F \cap \{x\colon\dist(x,\partial B_1) < \delta\}} \tilde\eta \cdot \n_{\partial^*F}
\frac{1}{(1{-}s)^{d-2}} \dH
\\&
= - \int_{\partial^*F \cap \{x\colon\dist(x,\partial B_1) < \delta\}} 
\tilde\eta \cdot \frac{1}{(1{-}s)^{d-2}}\Big(\mathrm{Id} {-} \nabla s \otimes \nabla s\Big)\n_{\partial^*F} \dH.
\end{aligned}
\end{equation}
Hence, for any Borel measurable $U \subset \partial B_1$ we have the representation formula
\begin{equation}
\begin{aligned}
\label{eq:formulaBVderivative}
&D^{tan}\tilde{h} (U) 
\\&~~~
=  \frac{1}{(1{-}s)^{d-2}}\Big(\mathrm{Id} {-} \nabla s \otimes \nabla s\Big) \n_{\partial^*F} \,\mathcal{H}^1
\big(\partial^*F \cap U^\delta \cap  \{x\colon\dist(x,\partial B_1) < \delta\}\big),
\end{aligned}
\end{equation}
where $U^\delta = \{x - s \frac{x}{|x|} \colon x \in U, s \in (-\delta,\delta)\}$.
This shows $\tilde h \in BV(\partial B_1)$ and thus $h \in BV(\mathbb{T}_\pi)$. 
More precisely, introducing the unit length tangent vectors $\ta_i(f(0,\vec{\theta})) := (\partial_{\theta_i}f)(0,\vec{\theta})$,
$\vec{\theta} = (\theta_1,\ldots,\theta_{d-1}) \in \mathbb{T}_\pi$, we obtain
\begin{align}
\label{eq:formulaBVderivative2}
\int_{\mathbb{T}_\pi} \zeta e_i \cdot \,\mathrm{d}Dh
= \int_{\partial B_1} \tilde\zeta \ta_i \cdot \,\mathrm{d}D^{tan}\tilde{h}
\end{align}
for all $\zeta \in C^\infty(\mathbb{T}_\pi)$ and $\tilde{\zeta}(f(0,\vec{\theta})):= \zeta(\vec{\theta})$,
$\vec{\theta}\in \mathbb{T}_\pi$, and all $i = 1,\ldots,d{-}1$.

\textit{Step~2 (Proof of~\emph{\eqref{eq:L2estimateDerivativeHeightL2}--\eqref{eq:L2estimateDerivativeHeight2}}).}
We deduce from~\eqref{eq:formulaBVderivative} that
\begin{align}
\|D^sh\|_{TV(\mathbb{T}_\pi)} \leq 
\mathcal{H}^1\big(\partial^*F \cap (\supp D^sh)^\delta \cap  \{x\colon\dist(x,\partial B_1) < \delta\}\big).
\end{align}
Since $\supp D^sh \subset \mathcal{N}$ and $\mathcal{L}^{d-1}(\mathcal{N}) = 0$, it further follows
\begin{equation}
\begin{aligned}
&\mathcal{H}^{d-1}\big(\partial^*F \cap (\supp D^sh)^\delta \cap  \{x\colon\dist(x,\partial B_1) < \delta\}\big)
\\&
\leq \mathcal{H}^{d-1}\bigg(\partial^*F \cap \mathcal{N}^\delta 
\cap  \{x\colon\dist(x,\partial B_1) < \delta\}
\cap \Big\{\n_{\partial^*F}(x)\cdot\n_{\partial B_1}\Big(\frac{x}{|x|}\Big) \neq 0\Big\}\bigg)
\\&~~~
+ \mathcal{H}^{d-1}\bigg(\partial^*F \cap \mathcal{N}^\delta 
\cap  \{x\colon\dist(x,\partial B_1) < \delta\}
\cap \Big\{\n_{\partial^*F}(x)\cdot\n_{\partial B_1}\Big(\frac{x}{|x|}\Big) = 0\Big\}\bigg)
\\&
= \int_{\mathbb{T}_\pi\cap\mathcal{N}}
\sum_{\substack{s \in \partial^*F_{\vec{\theta}} \\ \n_{\partial^*F}(f(s,\vec{\theta})) \cdot \n_{\partial B_1}(f(0,\vec{\theta})) \neq 0}}
\frac{(1{-}s)^{d{-}1}}{|\n_{\partial^*F}(f(s,\vec{\theta})) \cdot \n_{\partial B_1}(f(0,\vec{\theta}))|} \dvectheta
\\&~~~
+ \mathcal{H}^{d-1}\bigg(\partial^*F \cap \mathcal{N}^\delta 
\cap  \{x\colon\dist(x,\partial B_1) < \delta\}
\cap \Big\{\n_{\partial^*F}(x)\cdot\n_{\partial B_1}\Big(\frac{x}{|x|}\Big) = 0\Big\}\bigg)
\\&
= \mathcal{H}^{d-1}\bigg(\partial^*F \cap \mathcal{N}^\delta 
\cap  \{x\colon\dist(x,\partial B_1) < \delta\}
\cap \Big\{\n_{\partial^*F}(x)\cdot\n_{\partial B_1}\Big(\frac{x}{|x|}\Big) = 0\Big\}\bigg).
\end{aligned}
\end{equation}
Hence, \eqref{eq:L2estimateDerivativeHeight2} follows from the previous two displays.

For a proof of~\eqref{eq:L2estimateDerivativeHeightL2}
and~\eqref{eq:L2estimateDerivativeHeightL1}, we first record that
by~\eqref{eq:formulaBVderivative}, \eqref{eq:formulaBVderivative2}
and the notational shorthand $P_{Tan_{x}\partial B_1}$ for the orthogonal
projection onto~$Tan_{x}\partial B_1$, $x \in \partial B_1$,
\begin{equation}
\begin{aligned}
\label{eq:formulaAbsContPart}
\nabla h(\vec{\theta}) 
= \sum_{s \in \partial^*F_{\vec{\theta}}} 
\frac{1{-}s}{|\n_{\partial^*F}(f(s,\vec{\theta})) \cdot \n_{\partial B_1}(f(0,\vec{\theta}))|}
P_{Tan_{f(0,\vec{\theta})}\partial B_1} \big(\n_{\partial^*F}(f(s,\vec{\theta}))\big)
\end{aligned}
\end{equation}
for all $\vec{\theta} \in \mathbb{T}_\pi \setminus \mathcal{N}$. In particular,
\begin{align}
\label{eq:constantOutside}
\nabla h(\vec{\theta})  = 0, \quad \vec{\theta} \in \Theta_{\text{outside}}.
\end{align}

Furthermore, by the Pythagorean theorem,
\begin{equation}
\begin{aligned}
\label{eq:smallSlopeRegime}
q(s,\vec{\theta}) &:= \frac{\big|P_{Tan_{f(0,\vec{\theta})}\partial B_1} \big(\n_{\partial^*F}(f(s,\vec{\theta}))\big)\big|}
{|\n_{\partial^*F}(f(s,\vec{\theta})) \cdot \n_{\partial B_1}(f(0,\vec{\theta}))|} \leq c
\\&~~~~~~~~~~~~~~~~~~~
\;\Leftrightarrow\;
|\n_{\partial^*F}(f(s,\vec{\theta})) \cdot \n_{\partial B_1}(f(0,\vec{\theta}))| \geq \frac{1}{\sqrt{1{+}c^2}}.
\end{aligned}
\end{equation}
We can therefore perform the following case distinction for fixed $\vec{\theta} \in \Theta_{\text{graph}} \cup  \Theta_{\text{slope}}$
and fixed $s \in \partial^*F_{\vec{\theta}}$:
\begin{itemize}
\item[\textbf{1)}] If $\n_{\partial^*F}(f(s,\vec{\theta})) \cdot \n_{\partial B_1}(f(0,\vec{\theta})) < 0$, then due to
$e_{slope}(s,\theta) \geq 1-\delta^2$ and 
$|P_{Tan_{f(0,\vec{\theta})}\partial B_1} (\n_{\partial^*F}(f(s,\vec{\theta})))| \leq 1$
\begin{align}
(1{-}s) q(s,\vec{\theta}) \leq (1{+}O(\delta))
\frac{(1 {-} s)^{d{-}1}}{|\n_{\partial^*F}(f(s,\vec{\theta})) \cdot \n_{\partial B_1}(f(0,\vec{\theta}))|}
e_{slope}(s,\vec{\theta}).
\end{align}
\item[\textbf{2a)}] If $\vec{\theta} \in \Theta_{\text{Fuglede}} \subset \Theta_{\text{graph}}$,
in particular by definition $q(s,\vec{\theta}) \leq 1$, then due to
$|P_{Tan_{f(0,\vec{\theta})}\partial B_1} (\n_{\partial^*F}(f(s,\vec{\theta})))|^2 =
1 - |\n_{\partial^*F}(f(s,\vec{\theta})) \cdot \n_{\partial B_1}(f(0,\vec{\theta}))|^2
\leq 2 e_{slope}(s,\vec{\theta})$
and~\eqref{eq:smallSlopeRegime}
\begin{equation}
\begin{aligned}
&(1 {-} s)^2 q^2(s,\vec{\theta}) 
\\&~~~
\leq \sqrt{2} \frac{(1 {-} s)^2}{|\n_{\partial^*F}(f(s,\vec{\theta})) \cdot \n_{\partial B_1}(f(0,\vec{\theta}))|}
\Big(1 {-} \big(\n_{\partial^*F}(f(s,\vec{\theta})) \cdot \n_{\partial B_1}(f(0,\vec{\theta}))\big)^2\Big)
\\&~~~
\leq 2\sqrt{1{+}c^2} (1{+}O(\delta))
\frac{(1 {-} s)^{d{-}1}}{|\n_{\partial^*F}(f(s,\vec{\theta})) \cdot \n_{\partial B_1}(f(0,\vec{\theta}))|} 
e_{slope}(s,\vec{\theta}).
\end{aligned}
\end{equation}
\item[\textbf{2b)}] If $\vec{\theta} \in \Theta_{\text{slope}}$, 
$\n_{\partial^*F}(f(s,\vec{\theta})) \cdot \n_{\partial B_1}(f(0,\vec{\theta})) > 0$,
and $q(s,\vec{\theta}) \leq 1$, then choose first an adjacent $\tilde s \in \partial^*F_{\vec{\theta}} \setminus\{s\}$ such that
$\n_{\partial^*F}(f(\tilde s,\vec{\theta})) \cdot  \n_{\partial B_1}(f(0,\vec{\theta})) < 0$ and estimate
\begin{equation}
\begin{aligned}
&(1{-}s)q(s,\vec{\theta}) 
\\&~~~
\leq 1{+}O(\delta) 
\leq (1 {+} O(\delta)) \frac{(1{-}\tilde{s})^{d{-}1}}
{|\n_{\partial^*F}(f(\tilde s,\vec{\theta})) \cdot  \n_{\partial B_1}(f(0,\vec{\theta}))|}
e_{slope}(\tilde{s},\vec{\theta}).
\end{aligned}
\end{equation}
\item[\textbf{3)}] Finally, if $\n_{\partial^*F}(f(s,\vec{\theta})) \cdot \n_{\partial B_1}(f(0,\vec{\theta})) > 0$
					and $q(s,\vec{\theta}) > 1$, then due to 
					$|P_{Tan_{f(0,\vec{\theta})}\partial B_1} (\n_{\partial^*F}(f(s,\vec{\theta})))| \leq 1$
					and, by the equivalence~\eqref{eq:smallSlopeRegime}, $1{-}\frac{1}{\sqrt{1{+}c^2}} 
					\leq 1 - \n_{\partial^*F}(f(s,\vec{\theta})) \cdot \n_{\partial B_1}(f(0,\vec{\theta}))
					=e_{slope}(s,\vec{\theta})$ we obtain
\begin{equation}
\begin{aligned}
&(1{-}s)q(s,\vec{\theta}) 
\\&~~~
\leq \Big(1{-}\frac{1}{\sqrt{1{+}c^2}}\Big)^{-1}
\frac{1 {-} s}{|\n_{\partial^*F}(f(s,\vec{\theta})) \cdot \n_{\partial B_1}(f(0,\vec{\theta}))|}
\big(1 - \n_{\partial^*F}(f(s,\vec{\theta})) \cdot \n_{\partial B_1}(f(0,\vec{\theta}))\big)
\\&~~~
\leq \Big(1{-}\frac{1}{\sqrt{1{+}c^2}}\Big)^{-1} (1{+}O(\delta))
\frac{(1 {-} s)^{d{-}1}}{|\n_{\partial^*F}(f(s,\vec{\theta})) \cdot \n_{\partial B_1}(f(0,\vec{\theta}))|} 
e_{slope}(s,\vec{\theta}).
\end{aligned}
\end{equation}
\end{itemize}

Item~\textbf{2a)} with~\eqref{eq:formulaAbsContPart} immediately yields~\eqref{eq:L2estimateDerivativeHeightL2}.
Furthermore, we carefully note that for $\vec{\theta} \in \Theta_{\text{slope}}$
one can always arrange the choice of an adjacent~$\tilde s \in \partial^*F_{\vec{\theta}} \setminus\{s\}$ 
from item~\textbf{2b)} such that one uses the error control given by item~\textbf{1)} 
for that particular~$\tilde s$
at most $2$ times through an application of item~\textbf{2b)}, hence overall at most $3$ times.
In summary, combining items~\textbf{1)}, \textbf{2b)} and~\textbf{3)} 
with~\eqref{eq:formulaAbsContPart} and~\eqref{eq:constantOutside} yields~\eqref{eq:L2estimateDerivativeHeightL1}.

\textit{Step~3 (Proof of~\eqref{eq:L2estimateHeight}).}
First, we claim that for every $C' \in (1,\infty)$ one may choose
 $\delta \ll 1$ and $\eps \ll_\delta 1$ such that
\begin{align}
\int_{\Theta_{\text{outside}} \cup \Theta_{\text{slope}}} 
h^2 \dvectheta &\leq \frac{1}{C'} E_{rel}[F|\xi].
\end{align}
For the case of angles in~$\Theta_{\text{slope}}$, this follows
from the simple estimate
\begin{align}
\int_{\Theta_{\text{slope}}} h^2 \dvectheta
\leq \delta^2 \mathcal{L}^{d{-}1}(\Theta_{\text{slope}})
\end{align}
and the already established estimate~\eqref{eq:controlAmountOfOVerlaps}.
For the case of angles in~$\Theta_{\text{outside}}$, one may argue based on
\begin{equation}
\begin{aligned}
\int_{\Theta_{\text{outside}}} h^2 \dvectheta
&\lesssim \mathcal{L}^{d{-}1}(\Theta_{\text{outside}})\delta^2
\\&
\lesssim \delta \Big(\mathcal{L}^d(B_{1-\frac{\delta}{2}}\setminus F)
+ \mathcal{L}^d\big((\Rd\setminus B_{1+\frac{\delta}{2}}) \cap F\big)\Big)
\end{aligned}
\end{equation}
as well as an argument relying on the relative isoperimetric
inequality in~$B_{1-\frac{\delta}{2}}$ resp.\ $\Rd\setminus B_{1+\frac{\delta}{2}}$
analogous to the proof of Lemma~\ref{lem:BVgraphApprox}.

For the remaining contribution from~$\Theta_{\text{graph}}$, we simply observe
\begin{equation}
\begin{aligned}
&\int_{\Theta_{\text{graph}}} h^2 \dvectheta
\\&~~~
\leq (1{+}O(\delta))\int_{\Theta_{\text{graph}}} 
\frac{(1 {-} h(\vec{\theta}))^{d{-}1}}{|
\n_{\partial^*F}(f(h(\vec{\theta}),\vec{\theta})) \cdot \n_{\partial B_1}(f(0,\vec{\theta}))|}
e_{height}(h(\vec{\theta}),\vec{\theta})  \dvectheta.
\end{aligned}
\end{equation}
This concludes the proof.
\end{proof}

\subsection{From $BV$ graph approximation to smooth graph approximation}
In order to arrive at a situation where we will apply a Fuglede-type argument
to close the estimates, we need to regularize our $BV$ graph approximation.
To this end, let $\rho \in C^\infty_{cpt}(\Rd[];[0,1])$ be a standard radial symmetric mollifier 
with unit mass and support in~$\{x'\in\Rd[d-1]\colon |x'|\leq 1\}$, and let $\lambda \in (0,1)$
be an upper bound for the relative energy in the sense that
\begin{align}
E_{rel}[F|\xi] \leq \lambda^2 \ll 1.
\end{align}
We think of $\lambda$ as
\begin{align}
\label{eq:mollifierScale}
\lambda^{2(d{-}1)} := M^2E_{rel}[F|\xi],
\quad\text{with } M \gg 1 \text{ and } \eps \ll_M 1.
\end{align}
We then define a smooth height function $h_\lambda\colon \mathbb{T}_\pi \to [-\delta,\delta]$
by means of (again, we remind the reader that we identify $h$, cf.\ \eqref{def:height},
with its periodic extension from $\mathbb{T}_\pi = [0,\pi)^{d-2}{\times}[0,2\pi)$
to~$\mathbb{R}^{d-1}$)
\begin{align}
\label{eq:mollifiedHeigth}
h_\lambda(\vec{\theta}) := \int_{\Rd[d-1]} \frac{1}{\lambda^{d{-}1}}
\rho\Big(\frac{\vec{\theta}-\vec{\theta'}}{\lambda}\Big) h(\vec{\theta'}) \dvectildetheta.
\end{align}

\begin{lemma}
\label{lem:L2estimatesSmoothGraphApprox}
For each $C' \in (1,\infty)$ and $c \in (0,1]$ one may choose the constants $M \gg_{C',c} 1$, $\delta \ll 1$, 
and $\eps \ll_{\delta} 1$ such that for $\lambda$ from~\eqref{eq:mollifierScale}
with corresponding mollified height function~$h_\lambda$ from~\eqref{eq:mollifiedHeigth}
we have an $L^2$ estimate
for the height of the smooth graph approximation (recall the notation from~\ref{subsubsec:coercivity})
in the form of
\begin{equation}
\begin{aligned}
\label{eq:L2estimateSmoothHeight}
&\int_{\mathbb{T}_\pi} h_\lambda^2 \dvectheta 
\\&~~~
\leq \big(1{+}O(\delta)\big)  \int_{\mathbb{T}_\pi\setminus \mathcal{N}} 
\sum_{s \in \partial^*F_{\vec{\theta}}}  \frac{(1 {-} s)^{d{-}1}}{|
\n_{\partial^*F}(f(s,\vec{\theta})) \cdot \n_{\partial B_1}(f(0,\vec{\theta}))|}
e_{height}(s,\vec{\theta})  \dvectheta
\\&~~~~~~
+ \frac{1}{C'} E_{rel}[F|\xi],
\end{aligned}
\end{equation}
as well as an $L^2$ estimate for its derivative in the form of
\begin{equation}
\begin{aligned}
\label{eq:L2estimateDerivativeSmoothHeight}
&\int_{\mathbb{T}_\pi} |\nabla h_\lambda|^2 \dvectheta
\\&~~~
\leq \Big(2\sqrt{1{+}c^2} {+} O(\delta) {+} \frac{1}{C'}\Big) 
\\&~~~~~~~~~
\times
\int_{(\mathbb{T}_\pi\setminus \mathcal{N}) \cap \Theta_{\emph{\text{Fuglede}}}}  
\sum_{s \in \partial^*F_{\vec{\theta}}}  \frac{(1 {-} s)^{d{-}1}}{|
\n_{\partial^*F}(f(s,\vec{\theta})) \cdot \n_{\partial B_1}(f(0,\vec{\theta}))|}
e_{slope}(s,\vec{\theta}) \dvectheta
\\&~~~~~~
+ 2\int_{\mathbb{T}_\pi\setminus (\mathcal{N}  \cup \Theta_{\emph{\text{Fuglede}}})}  
\sum_{s \in \partial^*F_{\vec{\theta}}}  \frac{(1 {-} s)^{d{-}1}}{|
\n_{\partial^*F}(f(s,\vec{\theta})) \cdot \n_{\partial B_1}(f(0,\vec{\theta}))|}
e_{slope}(s,\vec{\theta}) \dvectheta
\\&~~~~~~
+ 2\mathcal{H}^{d-1}\bigg(\partial^*F \cap \{\dist(\cdot,\partial B_1) < \delta\} \cap \mathcal{N}^\delta
\cap \Big\{x \colon \n_{\partial^*F}(x)\cdot \n_{\partial B_1}\Big(\frac{x}{|x|}\Big)=0\Big\}\bigg).
\end{aligned}
\end{equation}
In other words, recalling~\eqref{eq:errorControlSlice} and~\eqref{eq:errorControlVertical},
the analogue of the sharp lower bound of~\eqref{eq:FugledeAux5} holds.
\end{lemma}

\begin{proof}[Proof of Lemma~\ref{lem:L2estimatesSmoothGraphApprox}]
Since~\eqref{eq:L2estimateSmoothHeight} follows directly from~\eqref{eq:L2estimateHeight},
we immediately turn our attention to the proof of~\eqref{eq:L2estimateDerivativeSmoothHeight}.

Define the convex function $G\colon[0,\infty) \to [0,\infty)$
\begin{align}
\label{eq:recessionFct}
G(\ell) := 
\begin{cases}
\ell^2, &\text{if } \ell \in [0,1], 
\\
2\ell - 1, &\text{if } \ell \in (1,\infty).
\end{cases}
\end{align}
Note that 
\begin{align}
\label{eq:generalTriangleInequ}
G(\ell {+} \tilde\ell) \leq 2G(\ell) + 2G(\tilde\ell),
\quad \ell,\tilde\ell \in [0,\infty),
\end{align}
and
\begin{align}
\label{eq:refinedTriangleInequ}
G(\ell {+} \tilde\ell)
\leq (1 {+} \kappa^2)\ell^2 + (1 {+} \kappa^{-2})\tilde\ell^2,
\quad \ell,\tilde\ell \in [0,1], \ell + \tilde \ell \leq 1, \kappa \in (0,1). 
\end{align}
Since 
\begin{align}
\label{eq:formulaDerivativeSmoothGraph}
\nabla h_\lambda(\theta) = \int_{\Rd[d-1]} \frac{1}{\lambda^{d{-}1}}
\rho\Big(\frac{\vec{\theta}-\vec{\theta'}}{\lambda}\Big) \nabla h(\vec{\theta'}) \dvectildetheta
+ \int_{\Rd[d-1]} \frac{1}{\lambda^{d{-}1}}
\rho\Big(\frac{\vec{\theta}-\vec{\theta'}}{\lambda}\Big)\,\mathrm{d}D^sh(\vec{\theta'}),
\end{align}
it follows from non-negativity of~$\rho$, monotonicity of~$G$, \eqref{eq:generalTriangleInequ}, 
and Jensen's inequality
\begin{equation}
\label{eq:LipschitzBoundAux1}
\begin{aligned}
&G(|\nabla h_\lambda|) 
\\&~~~
\leq G\bigg(\int_{\Rd[d-1]} \frac{1}{\lambda^{d{-}1}}
\rho\Big(\frac{\vec{\theta}-\vec{\theta'}}{\lambda}\Big) \big|\nabla h(\vec{\theta'})\big| \dvectildetheta
+ \int_{\Rd[d-1]} \frac{1}{\lambda^{d{-}1}}
\rho\Big(\frac{\vec{\theta}-\vec{\theta'}}{\lambda}\Big)\,\mathrm{d}\big|D^sh\big|(\vec{\theta'})\bigg)
\\&~~~
\leq \int_{\Rd[d-1]} 2G\big(\big|\nabla h(\vec{\theta'})\big|\big)
 \frac{1}{\lambda^{d{-}1}}\rho\Big(\frac{\vec{\theta}-\vec{\theta'}}{\lambda}\Big)  \dvectildetheta
\\&~~~~~~
+ 2 G\bigg(\int_{\Rd[d-1]} \frac{1}{\lambda^{d{-}1}}
\rho\Big(\frac{\vec{\theta}-\vec{\theta'}}{\lambda}\Big)\,\mathrm{d}\big|D^sh\big|(\vec{\theta'})\bigg).
\end{aligned}
\end{equation}
Since, by~\eqref{eq:L2estimateDerivativeHeight2},
\begin{align}
\label{eq:LipschitzBoundAux2}
\int_{\Rd[d-1]} \frac{1}{\lambda^{d{-}1}}
\rho\Big(\frac{\vec{\theta}-\vec{\theta'}}{\lambda}\Big)\,\mathrm{d}\big|D^sh\big|(\vec{\theta'})
\leq \frac{1}{\lambda^{d-1}}\int_{\partial^*F \cap \mathcal{N}^\delta} 
1 - \n_{\partial^*F} \cdot \xi \dH
\end{align}
and since by~\eqref{eq:L2estimateDerivativeHeightL2}--\eqref{eq:L2estimateDerivativeHeightL1} 
and $G(\ell) \leq 2\min\{\ell,\ell^2\}$
\begin{equation}
\begin{aligned}
\label{eq:LipschitzBoundAux3}
&\int_{\Rd[d-1]} 2G\big(\big|\nabla h(\vec{\theta'})\big|\big)
 \frac{1}{\lambda^{d{-}1}}\rho\Big(\frac{\vec{\theta}-\vec{\theta'}}{\lambda}\Big)  \dvectildetheta
\\&~~~
\leq 4\bigg(\max\Big\{3,\Big(1{-}\frac{1}{\sqrt{1{+}c^2}}\Big)^{-1}\Big\} {+} O(\delta)\bigg)
\\&~~~~~~~~~
\times
\frac{1}{\lambda^{d-1}}\int_{(\partial^*F \setminus \mathcal{N}^\delta) \cap \{\dist(\cdot,\partial B_1) < \delta\}} 
1 - \n_{\partial^*F} \cdot \xi \dH,
\end{aligned}
\end{equation}
we obtain for $M \gg_{c} 1$ the global Lipschitz bound
\begin{equation}
\label{eq:LipschitzBound}
\begin{aligned}
&G(|\nabla h_\lambda|) 
\\&~~~
\leq G\bigg(\int_{\Rd[d-1]} \frac{1}{\lambda^{d{-}1}}
\rho\Big(\frac{\vec{\theta}-\vec{\theta'}}{\lambda}\Big) \big|\nabla h(\vec{\theta'})\big| \dvectildetheta
+ \int_{\Rd[d-1]} \frac{1}{\lambda^{d{-}1}}
\rho\Big(\frac{\vec{\theta}-\vec{\theta'}}{\lambda}\Big)\,\mathrm{d}\big|D^sh\big|(\vec{\theta'})\bigg)
\\&~~~
\leq \sqrt{E_{rel}[F|\xi]} \leq 1,
\end{aligned}
\end{equation}
i.e., $|\nabla h_\lambda|^2=G(|\nabla h_\lambda|) \leq 1$.
Denoting by $\widetilde\Theta_{\text{Fuglede}} \subset \Rd[d-1]$ the
periodic extension of $\Theta_{\text{Fuglede}} \subset \mathbb{T}_\pi$ to $\Rd[d-1]$,
we capitalize on the previous estimate by using~\eqref{eq:refinedTriangleInequ}
instead of~\eqref{eq:generalTriangleInequ} to get, for all $\kappa \in (0,1)$, 
\begin{equation}
\label{eq:sobolevEstimateAux1}
\begin{aligned}
&\int_{\mathbb{T}_\pi} |\nabla h_\lambda|^2 \dvectheta
\\&~~~
\leq \int_{\mathbb{T}_\pi}
\int_{\widetilde\Theta_{\text{Fuglede}}} (1{+}\kappa^2)\big|\nabla h(\vec{\theta'})\big|^2
 \frac{1}{\lambda^{d{-}1}}\rho\Big(\frac{\vec{\theta}-\vec{\theta'}}{\lambda}\Big)  \dvectildetheta
\dvectheta
\\&~~~~~~
+ \int_{\mathbb{T}_\pi} (1{+}\kappa^{-2})
\bigg(\int_{\Rd[d-1]\setminus\widetilde\Theta_{\text{Fuglede}}}\big|\nabla h(\vec{\theta'})\big|
 \frac{1}{\lambda^{d{-}1}}\rho\Big(\frac{\vec{\theta}-\vec{\theta'}}{\lambda}\Big)  \dvectildetheta\bigg)^2
\dvectheta
\\&~~~~~~
+ \int_{\mathbb{T}_\pi} 
(1 {+} \kappa^{-2}) \bigg(\int_{\Rd[d-1]} \frac{1}{\lambda^{d{-}1}}
\rho\Big(\frac{\vec{\theta}-\vec{\theta'}}{\lambda}\Big)\,\mathrm{d}\big|D^sh\big|(\vec{\theta'})\bigg)^2
\dvectheta.
\end{aligned}
\end{equation}
Choosing $\kappa \ll_{C'} 1$ such that $(1{+}\kappa^2)2\sqrt{1{+}c^2} \leq 2\sqrt{1{+}c^2} + \frac{1}{C'}$, 
we thus obtain from~\eqref{eq:L2estimateDerivativeHeightL2}
\begin{equation}
\label{eq:sobolevEstimateAux2}
\begin{aligned}
&\int_{\mathbb{T}_\pi}
\int_{\widetilde\Theta_{\text{Fuglede}}} (1{+}\kappa^2)\big|\nabla h(\vec{\theta'})\big|^2
 \frac{1}{\lambda^{d{-}1}}\rho\Big(\frac{\vec{\theta}-\vec{\theta'}}{\lambda}\Big)  \dvectildetheta
\dvectheta
\\&~~~
\leq \text{ first r.h.s.\ term of }\eqref{eq:L2estimateDerivativeSmoothHeight}.
\end{aligned}
\end{equation}
Furthermore, we may choose $M \gg_{C',c} 1$ such that~\eqref{eq:L2estimateDerivativeHeightL1}
and~\eqref{eq:errorControlSlice} imply
\begin{equation}
\label{eq:sobolevEstimateAux4}
\begin{aligned}
&\int_{\mathbb{T}_\pi} (1{+}\kappa^{-2})
\bigg(\int_{\Rd[d-1]\setminus\widetilde\Theta_{\text{Fuglede}}}\big|\nabla h(\vec{\theta'})\big|
 \frac{1}{\lambda^{d{-}1}}\rho\Big(\frac{\vec{\theta}-\vec{\theta'}}{\lambda}\Big)  \dvectildetheta\bigg)^2
\dvectheta
\\&~~~
\leq \mathcal{L}^{d{-}1}(\mathbb{T}_\pi)(1{+}\kappa^{-2})\frac{1}{M^2E_{rel}[F|\xi]}
\bigg(\int_{\mathbb{T}_\pi \setminus \Theta_{\emph{\text{Fuglede}}}} |\nabla h| \dvectheta\bigg)^2
\\&~~~
\leq \text{ second r.h.s.\ term of }\eqref{eq:L2estimateDerivativeSmoothHeight}.
\end{aligned}
\end{equation}
Finally, by choosing $M \gg_{C'} 1$ it follows from~\eqref{eq:LipschitzBoundAux2}, \eqref{eq:L2estimateDerivativeHeight2}
and $\mathcal{H}^{d-1}(\partial^*F \cap \{\dist(\cdot,\partial B_1) < \delta\} \cap \mathcal{N}^\delta
\cap \{x \colon \n_{\partial^*F}(x)\cdot \n_{\partial B_1}(\frac{x}{|x|})=0\}) \leq E_{rel}[F|\xi]$
\begin{equation}
\label{eq:sobolevEstimateAux3}
\begin{aligned}
&\int_{\mathbb{T}_\pi} 
(1 {+} \kappa^{-2}) \bigg(\int_{\Rd[d-1]} \frac{1}{\lambda^{d{-}1}}
\rho\Big(\frac{\vec{\theta}-\vec{\theta'}}{\lambda}\Big)\,\mathrm{d}\big|D^sh\big|(\vec{\theta'})\bigg)^2
\dvectheta
\\&~~~
\leq \mathcal{L}^{d{-}1}(\mathbb{T}_\pi)(1{+}\kappa^{-2})\frac{1}{M^2E_{rel}[F|\xi]}
\\&~~~~~~
\times \mathcal{H}^{d-1}\bigg(\partial^*F \cap \{\dist(\cdot,\partial B_1) < \delta\} \cap \mathcal{N}^\delta
\cap \Big\{x \colon \n_{\partial^*F}(x)\cdot \n_{\partial B_1}\Big(\frac{x}{|x|}\Big)=0\Big\}\bigg)^2
\\&~~~
\leq \text{ third r.h.s.\ term of }\eqref{eq:L2estimateDerivativeSmoothHeight}.
\end{aligned}
\end{equation}
Hence, \eqref{eq:L2estimateDerivativeSmoothHeight} follows from the previous three estimates.
\end{proof}

The estimates from Lemma~\ref{lem:L2estimatesBVgraphApprox} and
Lemma~\ref{lem:L2estimatesSmoothGraphApprox} allow us to pass from
the $BV$ graph approximation~$\chi_h$ of~$\chi_F$ to the smooth graph approximation~$\chi_{h_\lambda}$
of~$\chi_F$ defined by
\begin{align}
\chi_{h_{\lambda}}(f(s,\vec{\theta})) := \chi_{B_1}(f(s,\vec{\theta})) 
												- \chi_{\{(s,\vec{\theta})\colon 0 \leq s \leq h_{\lambda}(\vec{\theta})\}}(s,\vec{\theta})
												+ \chi_{\{(s,\vec{\theta})\colon h_{\lambda}(\vec{\theta}) \leq s \leq 0\}}(s,\vec{\theta}).
\end{align}

\begin{lemma}
\label{lem:smoothGraphApprox}
For each $C' \in (1,\infty)$ one may choose
$M \gg 1$, $\delta \ll 1$, and $\eps \ll_{M} 1$ such that for $\lambda$ from~\eqref{eq:mollifierScale}
with corresponding mollified height function~$h_\lambda$ from~\eqref{eq:mollifiedHeigth}
it holds
\begin{align}
\label{eq:fromBVtoSmoothGraphApproximation}
-\int_{\Rd} (\chi_h - \chi_{h_\lambda}) (\nabla\cdot\xi + (d{-}1)) \dx 
\geq - \frac{1}{C'} E_{rel}[F|\xi].
\end{align}
\end{lemma}

\begin{proof}[Proof of Lemma~\ref{lem:smoothGraphApprox}]
We change variables and make use of~\eqref{eq:VolumeIntegrand} to compute
\begin{equation}
\begin{aligned}
\label{eq:fromBVtoSmoothAux1}
&\bigg|\int_{\Rd} (\chi_h - \chi_{h_\lambda}) (\nabla\cdot\xi + 1) \dx \bigg|
\\&~~~
\leq \int_{\mathbb{T}_\pi} \bigg|\int_{h}^{h_\lambda} 
(1{-}s)^{d-1}\big(d{+}1)s
\ds\bigg| \dvectheta.
\end{aligned}
\end{equation}
Since $|h|\leq \delta$ and $|h_\lambda|\leq \delta$, we may estimate based
on the $L^2$ estimates~\eqref{eq:L2estimateHeight} and~\eqref{eq:L2estimateSmoothHeight}
\begin{align}
\label{eq:fromBVtoSmoothAux2}
\bigg|\int_{\Rd} (\chi_h - \chi_{h_\lambda}) (\nabla\cdot\xi + 1) \dx \bigg|
\lesssim \int_{\mathbb{T}_\pi} \big|h_\lambda^2 - h^2\big| \dvectheta + O(\delta) E_{rel}[F|\xi].
\end{align}
Furthermore, by $x^2-y^2 = (x-y)(x+y)$, the Cauchy--Schwarz inequality, again the $L^2$
estimates~\eqref{eq:L2estimateHeight} as well as~\eqref{eq:L2estimateSmoothHeight}, and~\eqref{eq:errorControlSlice}
\begin{align}
\label{eq:fromBVtoSmoothAux3}
\int_{\mathbb{T}_\pi} \big|h_\lambda^2 - h^2\big| \dvectheta
\lesssim \bigg(\int_{\mathbb{T}_\pi} (h_\lambda - h)^2 \dvectheta\bigg)^\frac{1}{2} \sqrt{E_{rel}[F|\xi]}.
\end{align}
Smuggling in a telescope sum, we obtain from the triangle inequality
\begin{align}
\bigg(\int_{\mathbb{T}_\pi} (h_\lambda - h)^2 \dvectheta\bigg)^\frac{1}{2} \leq
\sum_{k=0}^\infty \bigg(\int_{\mathbb{T}_\pi} \big(h_{2^{-(k{+}1)}\lambda} - h_{2^{-k}\lambda}\big)^2 \dvectheta\bigg)^\frac{1}{2}.
\end{align}
A $BV$ slicing argument with respect to a suitable hyperplane in~$\Rd[d{-}1]$ 
moreover yields
\begin{align}
\label{eq:aux99a}
\sup_{\vec{\theta}\in\mathbb{T}_\pi} 
\big|h_{2^{-(k{+}1)}\lambda}(\vec{\theta}) - h_{2^{-k}\lambda}(\vec{\theta})\big|
\lesssim 2^{-k}\lambda \|Dh\|_{TV(\mathbb{T}_\pi)}.
\end{align}
Together with the estimate
\begin{equation}
\begin{aligned}
\label{eq:aux99b}
&\|Dh\|_{TV(\mathbb{T}_\pi)}
\\&~~~
\leq \int_{\mathbb{T}_\pi} |\nabla h| \chi_{|\nabla h|\geq 1} \dvectheta
 + \sqrt{\mathcal{L}^{d-1}(\mathbb{T}_\pi)} \bigg( \int_{\mathbb{T}_\pi} |\nabla h|^2 \chi_{|\nabla h|\leq 1} \dvectheta\bigg)^\frac{1}{2}
 + \|D^sh\|_{TV(\mathbb{T}_\pi)},
\end{aligned}
\end{equation}
it follows from~\eqref{eq:L2estimateDerivativeHeightL2}--\eqref{eq:L2estimateDerivativeHeight2}, 
\eqref{eq:mollifierScale}, and the previous estimates
\begin{align}
\int_{\mathbb{T}_\pi} \big|h_\lambda^2 - h^2\big| \dvectheta
\lesssim M^\frac{1}{d{-}1}E_{rel}[F|\xi]^\frac{1}{2(d{-}1)}E_{rel}[F|\xi].
\end{align}
Plugging this estimate back into~\eqref{eq:fromBVtoSmoothAux2} yields the claim.
\end{proof}

\subsection{The Fuglede argument for the smooth graph approximation}
We move on by estimating the error between the smooth graph approximation~$\chi_{h_\lambda}$
of the fixed competitor~$F$ and the minimizer~$B_1$.

\begin{lemma}
\label{lem:Fuglede}
For each $\gamma \in (0,1)$ one may choose~$M \gg 1$, $\delta \ll_M 1$, $\eps \ll_{\delta,M} 1$, $a \in \Rd[]$ and~$\vec{b} \in \Rd$
such that for $\lambda$ from~\eqref{eq:mollifierScale}
with corresponding mollified height function~$h_\lambda$ from~\eqref{eq:mollifiedHeigth}
it holds
\begin{equation}
\begin{aligned}
\label{eq:Fuglede}
&-\int_{\Rd} (\chi_{h_\lambda} - \chi_{B_1}) (\nabla\cdot\xi + 1) \dx
- \int_{\Rd} (\chi_{h_\lambda} - \chi_{B_1}) L_{a,\vec{b}} \dx
\\&~~~~~~~~~~~
\geq - (1{+}\gamma)\frac{1}{2} E_{rel}[F|\xi].
\end{aligned}
\end{equation}
as well as
\begin{equation}
\begin{aligned}
\label{eq:Fuglede2}
&-\int_{\Rd} (\chi_F - \chi_{h_\lambda}) L_{a,\vec{b}} \dx
\geq - \gamma E_{rel}[F|\xi].
\end{aligned}
\end{equation}
\end{lemma}

\begin{proof}[Proof of Lemma~\ref{lem:Fuglede}]
Recalling Remark~\ref{rem:proofClassicalFuglede}, the bounds from Lemma~\ref{lem:L2estimatesSmoothGraphApprox}
precisely allow us to prove~\eqref{eq:Fuglede} based on the same arguments as presented in Section~\ref{subsec:classicalFuglede}
for the proof of the classical Fuglede result~\eqref{eq:FugledeAux2b}, but now of course under the substantially weaker assumption of $E_{rel}[F|\xi] \ll 1$.

It remains to prove~\eqref{eq:Fuglede2}. To this end, it is important to recall from comparing with~\eqref{eq:choiceOfShifts} that
\begin{align}
a = c_0 \int_{\mathbb{T}_\pi} h_\lambda(\vec{\theta}) \dvectheta \text{ and } 
\vec{b} = c_1 \int_{\mathbb{T}_\pi} h_\lambda(\vec{\theta}) f(0,\vec{\theta}) \dvectheta,
\end{align}
so that in particular by~\eqref{eq:L2estimateSmoothHeight} and~\eqref{eq:errorControlSlice} it follows
\begin{align}
\label{eq:estimatesShifts}
|a| + |\vec{b}| \lesssim \sqrt{E_{rel}[F|\xi]}. 
\end{align}
This in turn entails
\begin{align}
\bigg|\int_{\Rd} (\chi_{h} - \chi_{h_\lambda}) L_{a,\vec{b}} \dx\bigg|
\lesssim \sqrt{E_{rel}[F|\xi]} \int_{\mathbb{T}_\pi} |h - h_\lambda| \dvectheta.
\end{align}
A telescope bound together with~\eqref{eq:aux99a}, \eqref{eq:aux99b}, \eqref{eq:L2estimateDerivativeHeightL2}--\eqref{eq:L2estimateDerivativeHeight2}, 
and \eqref{eq:mollifierScale} thus guarantees
\begin{align}
\label{eq:aux98a}
\bigg|\int_{\Rd} (\chi_{h} - \chi_{h_\lambda}) L_{a,\vec{b}} \dx\bigg|
\lesssim M^\frac{1}{d{-}1}E_{rel}[F|\xi]^\frac{1}{2(d{-}1)}E_{rel}[F|\xi],
\end{align}
which is a bound of required form.

To control the contribution from the difference between the competitor~$F$ and its
$BV$ graph approximation, we observe that, 
due to the assumption from~\eqref{eq:diameterConstraint} that $F\subset B_R$%
\footnote{%
This is the only time we make use of the diameter constraint~\eqref{eq:diameterConstraint}
in the proof of Theorem~\ref{theo:pertRegime}: it heals the fact that $x \mapsto |x| \notin L^1(\Rd)$.
Following Figalli and Zhang~\cite{figalli2026sharpstabilityalexandrovstheorem}, one
may, as an alternative to our assumption~\eqref{eq:barycenterConstraint2}, work with a weighted barycenter centering of the finite perimeter
competitor~$F$ in the form of $\int_F \frac{x}{|x|}\dx = 0$ and obtain the same conclusions
without requiring~\eqref{eq:diameterConstraint}. Indeed, this does not change leading
order computations in the perturbative smooth graph regime.
},
we may simply employ the argument from the proof of Lemma~\ref{lem:BVgraphApprox}
to deduce that for every $\gamma \in (0,1)$ one may choose $\delta \ll 1$
and $\eps \ll_{\delta,R} 1$ such that
\begin{align}
\label{eq:aux98b}
\bigg|\int_{\Rd} (\chi_{F} - \chi_{h}) L_{a,\vec{b}} \dx\bigg|
\leq \gamma E_{rel}[F|\xi].
\end{align}

In summary, \eqref{eq:Fuglede2} follows from~\eqref{eq:aux98a} and~\eqref{eq:aux98b}.
\end{proof}

\subsection{Proof of Theorem~\ref{theo:pertRegime}}\label{subsec:proofTheoremPertRegime} 
We complete the
\begin{proof}[Proof of Theorem~\ref{theo:pertRegime}]
It remains to choose for given $\gamma \in (0,1)$ the constants $M \gg 1$, $\delta \ll_M 1$  
and~$\eps \ll_{\delta,M} 1$ such that Lemma~\ref{lem:BVgraphApprox}
and Lemma~\ref{lem:smoothGraphApprox} guarantee
\begin{align}
-\int_{\Rd} (\chi_F - \chi_{h_\lambda}) (\nabla\cdot\xi + 1) \dx
\geq - \gamma E_{rel}[F|\xi]. 
\end{align}
Inserting this estimate as well as the Fuglede-type estimates
from Lemma~\ref{lem:Fuglede} back into our relative
energy equality~\eqref{eq:relativeEnergyEquality}
establishes~\eqref{eq:quantIsoperInequ2}.
\end{proof}

\subsection{Proof of Theorem~\ref{theo:globalRegime}}
\label{subsec:proofMainTheorem}

We proceed in two steps, the combination of which yielding
the claim of Theorem~\ref{theo:globalRegime}.

\textit{Step~1 (Reduction to bounded sets).} We claim that there
exist $R \in [2,\infty)$ and $C \in (1,\infty)$ such that for
all sets of finite perimeter $F \subset \Rd$ with $\mathcal{L}^d(F)=\omega_d$
we may construct another set of finite perimeter $F' \subset \Rd$ such that
\begin{align}
F' &\subset B_R,
\\
\mathcal{L}^d(F') &= \omega_d,
\\
\inf_{x_0 \in \Rd} E_{rel}[F{-}x_0|\xi] &\leq
\inf_{x_0 \in \Rd} E_{rel}[F'{-}x_0|\xi]
\\&~~~\notag
+ C\big(\mathcal{H}^{d{-}1}(\partial^* F) - \mathcal{H}^{d{-}1}(\partial B_1)\big),
\\
\mathcal{H}^{d{-}1}(\partial^* F') - \mathcal{H}^{d{-}1}(\partial B_1)
&\leq C\big(\mathcal{H}^{d{-}1}(\partial^* F) - \mathcal{H}^{d{-}1}(\partial B_1)\big).
\end{align}

This reduction step essentially follows from the existing literature;
in our particular case of a tilt-excess type distance from \cite{FuscoJulin14},
which itself is based on~\cite{Fusco2008}. For the former, a careful
inspection of the argument shows that we only need to deal
with a generalization of~\cite[inequality~(3.5), page~931]{FuscoJulin14}.

The relevant setting is the following: suppose we are given a set of finite
perimeter~$F$ as in the claim as well as a second set of finite perimeter
$\tilde{F} \subset F$ such that $\|\chi_F-\chi_{\tilde{F}}\|_{L^1(\Rd)}
= \mathcal{L}^d(F) - \mathcal{L}^d(\tilde{F}) \lesssim_d \mathcal{H}^{d{-}1}(\partial^* F) - \mathcal{H}^{d{-}1}(\partial B_1)$;
cf.\ \cite[inequality~(3.4), page~930]{FuscoJulin14}. Let $\tilde x$ be such that
$E_{rel}[\tilde{F}{-}\tilde{x}|\xi] = \inf_{x_0 \in \Rd} E_{rel}[\tilde{F}{-}x_0|\xi]$.
We then estimate
\begin{equation}
\begin{aligned}
&\inf_{x_0 \in \Rd} E_{rel}[F{-}x_0|\xi]
- \inf_{x_0 \in \Rd} E_{rel}[\tilde{F}{-}x_0|\xi]
\\&~~~~~~~~~
\leq E_{rel}[F{-}\tilde{x}|\xi] - E_{rel}[\tilde{F}{-}\tilde{x}|\xi]
\\&~~~~~~~~~
= \mathcal{H}^{d{-}1}(\partial^* F) - \mathcal{H}^{d{-}1}(\partial B_1)
\\&~~~~~~~~~~~~
+ \mathcal{H}^{d{-}1}(\partial B_1) - \mathcal{H}^{d{-}1}(\partial^*\tilde{F})
\\&~~~~~~~~~~~~
+ \int_{\Rd} \big(\chi_{F{-}\tilde{x}}-\chi_{\tilde{F}{-}\tilde{x}}\big) \nabla\cdot\xi \dx.
\end{aligned}
\end{equation}
Since the last right hand side term may be estimated by $\|\nabla\cdot\xi\|_{L^\infty}
\|\chi_F-\chi_{\tilde{F}}\|_{L^1(\Rd)}\lesssim_d \mathcal{H}^{d{-}1}(\partial^* F) - \mathcal{H}^{d{-}1}(\partial B_1)$,
we infer the claim of \textit{Step~1} as in the proof of~\cite[Lemma~3.2]{FuscoJulin14}. 

\textit{Step~2 (Contradiction by compactness).} 
Let $R \in [2,\infty)$ be the radius to which the conclusions of \textit{Step~1} apply.
We claim that there exists $C \in (1,\infty)$ such that for all sets of finite perimeter~$F \subset \Rd$
satisfying~\eqref{eq:volumeConstraint2} and~\eqref{eq:diameterConstraint}
it holds
\begin{align}
\label{eq:globalByCompactness}
\frac{1}{C} \inf_{x_0 \in \Rd} E_{rel}[F{-}x_0|\xi]
\leq \mathcal{H}^{d{-}1}(\partial^* F) - \mathcal{H}^{d{-}1}(\partial B_1).
\end{align}

We argue by contradiction and assume~\eqref{eq:globalByCompactness} does not hold.
In particular, for each $k \in \mathbb{N}$ we may choose a set of finite perimeter
$F_k \in \Rd$ satisfying~\eqref{eq:volumeConstraint2} and~\eqref{eq:diameterConstraint}
such that in addition
\begin{align}
\label{eq:globalByCompactnessAux1}
\mathcal{H}^{d{-}1}(\partial^* F_k) < \mathcal{H}^{d{-}1}(\partial B_1) +
\frac{1}{k} \inf_{x_0 \in \Rd} E_{rel}[F_k{-}x_0|\xi].
\end{align}
It immediately follows for $x_k := \frac{1}{\omega_d} \int_{F} x \dx$ that
\begin{align}
\label{eq:globalByCompactnessAux2}
\mathcal{H}^{d{-}1}\big(\partial^* (F_k {-} x_k)\big) < \mathcal{H}^{d{-}1}(\partial B_1) +
\frac{1}{k} E_{rel}[F_k{-}x_k|\xi].
\end{align}
Since by construction $E_{rel}[F_k{-}x_k|\xi] \leq \mathcal{H}^{d{-}1}\big(\partial^*(F_k {-} x_k)\big)
+ \omega_d\|\nabla\cdot\xi\|_{L^\infty(\Rd)}$, we may upgrade the previous inequality to
\begin{align}
\label{eq:globalByCompactnessAux3}
\Big(1 {-} \frac{1}{k}\Big)\mathcal{H}^{d{-}1}\big(\partial^*(F_k {-} x_k)\big) < \mathcal{H}^{d{-}1}(\partial B_1) +
\frac{1}{k} \omega_d\|\nabla\cdot\xi\|_{L^\infty(\Rd)}.
\end{align}
Hence, we may find a sub-sequence $(k_l)_{l \in \mathbb{N}}$
and a set of finite perimeter $F \subset B_R$ such that
$\|\chi_F-\chi_{F_{k_l}{-}x_{k_l}}\|_{L^1(\Rd)} \to 0$ as $l \to \infty$,
in particular $\mathcal{L}^d(F)=\omega_d$, $\int_{F} x\dx = 0$, as well as
$\mathcal{H}^{d{-}1}(\partial^* F) \leq \lim_{l\to\infty}\mathcal{H}^{d{-}1}(\partial^* F_{k_l})$.
From that, we furthermore deduce
\begin{align}
\label{eq:globalByCompactnessAux4}
\mathcal{H}^{d{-}1}(\partial B_1) \leq \mathcal{H}^{d{-}1}(\partial^* F)
\leq \lim_{l\to\infty}\mathcal{H}^{d{-}1}(\partial^* F_{k_l}) \leq \mathcal{H}^{d{-}1}(\partial B_1),
\end{align}
where the first inequality is a consequence of the qualitative isoperimetric inequality,
and where the last inequality follows from~\eqref{eq:globalByCompactnessAux3}.
We deduce $\mathcal{H}^{d{-}1}(\partial B_1) = \mathcal{H}^{d{-}1}(\partial^* F)
= \lim_{l\to\infty}\mathcal{H}^{d{-}1}(\partial^* F_{k_l})$, so that $B_1 = F$ and
\begin{align}
\label{eq:globalByCompactnessAux5}
0 = E_{rel}[F|\xi] = \lim_{l \to \infty} E_{rel}[F_{k_l}{-}x_{k_l}|\xi].
\end{align}
In summary, we obtain a contradiction to Theorem~\ref{theo:pertRegime}
(with $R$ replaced by $2R$).
\qed

\subsection{Proof of Lemma~\ref{lem:coercivity}}
\label{subsec:proofCoercivity}

We again proceed in two steps.

\textit{Step~1 (Proof of~\eqref{eq:controlFraenkel}).}
Smuggling in the $BV$ graph approximation~\eqref{eq:BVgraphApprox} to~$\chi_F$,
we deduce from the triangle inequality that
\begin{align}
\big|\mathcal{L}^d(F \Delta B_1)\big|^2
\leq 4 \bigg(\int_{\Rd} |\chi_F - \chi_{h}| \dx\bigg)^2
+ 4 \bigg(\int_{\Rd} |\chi_{h} - \chi_{B_1}| \dx\bigg)^2.
\end{align}
The first term yields a bound of required type by the argument for the proof
of Lemma~\ref{lem:BVgraphApprox}, whereas the second term as well yields
a bound of required type by the $L^2$-control~\eqref{eq:L2estimateHeight}
and the coercivity property~\eqref{eq:errorControlSlice}.

\textit{Step~2 (Proof of~\eqref{eq:tiltExcessControl}).} 
Recall that $\xi = \max\{1{-}s^2,0\}\nabla s$, where $s$ denotes the signed distance
to~$\partial B_1$ such that $\nabla s(x) = \n_{\partial B_1}(\frac{x}{|x|}) =  - \frac{x}{|x|}$.
We first recognize
\begin{align}
\frac{1}{2}\Big|\n_{\partial^*F}(x)
- \Big(-\frac{x}{|x|}\Big)\Big|^2
\leq 1 - \n_{\partial^*F}(x) \cdot  \n_{\partial B_1}\Big(\frac{x}{|x|}\Big).
\end{align}
In case $\n_{\partial^* F}(x) \cdot \n_{\partial B_1}(\frac{x}{|x|}) < 0$ we have
\begin{equation}
\begin{aligned}
1 - \n_{\partial^*F}(x) \cdot  \n_{\partial B_1}\Big(\frac{x}{|x|}\Big)
\leq 2 \leq 2\big(1 - \n_{\partial^* F}(x)\cdot\xi(x)\big),
\end{aligned}
\end{equation}
whereas in case $\n_{\partial^* F}(x) \cdot \n_{\partial B_1}(\frac{x}{|x|}) \geq 0$
it holds due to $\max\{1{-}s^2,0\} \leq 1$
\begin{equation}
\begin{aligned}
1 - \n_{\partial^*F}(x) \cdot  \n_{\partial B_1}\Big(\frac{x}{|x|}\Big)
\leq 1 - \n_{\partial^* F}(x)\cdot\xi(x).
\end{aligned}
\end{equation}
The previous three displays thus imply~\eqref{eq:tiltExcessControl}. 
\qed


\bibliographystyle{abbrv}
\bibliography{quantitative_isoperimetric_inequality}

\end{document}